\documentclass[a4paper,11pt]{article}

\usepackage[utf8]{inputenc}
\usepackage[T1]{fontenc}
\usepackage[english]{babel}
\usepackage{lmodern}
\usepackage{amsmath}
\usepackage{amssymb}
\usepackage{mathrsfs}
\usepackage{amsthm}
\usepackage{color}
\usepackage{soul}
\usepackage{diagbox}
\usepackage{ulem}
\usepackage{verbatim}
\usepackage{moreverb}
\usepackage{url}
\usepackage{graphicx}
\usepackage{listings}
\usepackage[breaklinks]{hyperref}
\usepackage{layout}
\usepackage[top=3cm, bottom=3cm, left=2.5cm, right=2.2cm]{geometry}
\hypersetup{urlcolor=blue,linkcolor=blue,citecolor=blue,colorlinks=true}

\newtheorem{theo}{Theorem}[section]
\newtheorem{defi}{Definition}[subsection]
\newtheorem{lem}{Lemma}[subsection]
\newtheorem{prop}{Proposition}[subsection]
\newtheorem{Rem}{Remark}[section]
\newtheorem{Cor}{Corollary}[section]
        
\title{\textbf{Asymptotic distribution of least squares estimators for linear models with dependent errors : regular designs}}    
   
\author{Emmanuel Caron \footnote{Emmanuel Caron, Ecole Centrale Nantes, Laboratoire de Mathématiques Jean Leray UMR 6629, 1 Rue de la Noë, 44300 Nantes. Email: emmanuel.caron@ec-nantes.fr}, Sophie Dede \footnote{Sophie Dede, Lycée Stanislas, 22 Rue Notre-Dame-des-Champs, 75006 Paris. Email: dede.sophie@gmail.com}}     
        
\begin{document}

\maketitle

\pagestyle{plain}

\begin{abstract}
In this paper, we consider the usual linear regression model in the case where the error process is assumed strictly stationary. We use a result from Hannan, who proved a Central Limit Theorem for the usual least squares estimator under general conditions on the design and the error process. We show that for a large class of designs, the asymptotic covariance matrix is as simple as the independent and identically distributed case. We then estimate the covariance matrix using an estimator of the spectral density whose consistency is proved under very mild conditions.
\end{abstract}

\section{Introduction}

We consider the usual fixed-design linear regression model: 
\[Y = X\beta + \epsilon,\]
where $X$ is the fixed design matrix and $(\epsilon_{i})_{i \in \mathbb{Z}}$ is a stationary process. This model is commonly used in time series regression.

Our work is based on the paper by Hannan \cite{hannan73clt}, who proved a Central Limit Theorem for the usual least squares estimator under general conditions on the design and the error process.
Most of short-range dependent processes satisfy the conditions on the error process, for instance the class of linear processes with summable coefficients and square integrable innovations, a large class of functions of linear processes, and many processes under various mixing conditions (see for instance \cite{dmv2007weak}, and also \cite{dedecker2015optimality} for the optimality of Hannan’s condition).

In this paper, it is shown that for a large class of designs satisfying the conditions of Hannan, the covariance matrix of the limiting distribution of the least squares estimator is the same as in the i.i.d.\footnote{independent and identically distributed.} case, up to the usual error variance term, which should be replaced by the covariance series of the error process.
We shall refer to this very large class of designs as « regular designs » (see Section $2.3$ for the precise definition). 
It includes many interesting examples, for instance the ANOVA type designs or the designs whose columns are regularly varying (such as the polynomial regression type designs).

For this class of regular designs, any consistent estimator of the covariance series of $(\epsilon_{i})_{i \in \mathbb{Z}}$ may be used to obtain a Gaussian limit distribution with explicit covariance matrix for the normalized least squares estimator. Doing so, it is then possible to obtain confidence regions and test procedures for the unknown parameter $\beta$. In this paper, assuming only that Hannan's condition on $(\epsilon_{i})$ is satisfied, we propose a consistent estimator of the spectral density of $(\epsilon_{i})$ (as a byproduct, we get an estimator of the covariance series).

Wu and Liu \cite{wuspectraldensity} considered the problem of estimating the spectral density for a large class of short-range dependent processes. They proposed a consistent estimator for the spectral density, and gave some conditions under which the centered estimator satisfies a Central Limit Theorem.
These results are based on the asymptotic theory of stationary processes developed by Wu \cite{wudependence}. 
This framework enables to deal with most of the statistical procedures from time series, including the estimation of the spectral density. However the class of processes satisfying the  $\mathbb{L}^{2}$ "physical dependence measure" introduced by Wu is included in the class of processes satisfying Hannan’s condition.
In this paper, we prove the consistency of an estimator of the spectral density of the error process under Hannan’s condition.
Compared to Wu’s precise results on the estimation of the spectral density (Central Limit Theorem, rates of convergence, deviation inequalities), our result is only a consistency result, but it holds under Hannan’s condition, that is for most of short-range dependent processes.

The paper is organized as follows. In Section $2$, we recall Hannan’s Central Limit Theorem for the least squares estimator, and we define the class of « regular designs » (we also give many examples of such designs).
In Section $3$, we focus on the estimation of the spectral density of the error process under Hannan's condition. 
Finally, some examples of stationary processes satisfying Hannan's condition are presented in Section $4$. 

\section{Hannan's theorem and regular design}

\subsection{Notations and definitions}

Let us recall the equation of the linear regression model: 
\begin{equation}
Y = X\beta + \epsilon,
\label{-1}
\end{equation}
where $X$ is a deterministic design matrix and $\epsilon$ is an error process defined on a probability space ($\Omega, \mathcal{F}, \mathbb{P}$). Let $X_{.,j}$ be the column $j$ of the matrix $X$, and $x_{i,j}$ the real number at the row $i$ and the column $j$, where $j$ is in $\{1, \ldots, p\}$ and $i$ in $\{1, \ldots, n\}$. The random vectors $Y$ and $\epsilon$ belong to $\mathbb{R}^{n}$ and $\beta$ is a $p \times 1$ vector of unknown parameters.

Let $\left \| . \right \|_{2}$ be the usual euclidean norm on $\mathbb{R}^{n}$, and $\left \| . \right \|_{\mathbb{L}^{p}}$ be the $\mathbb{L}^{p}$-norm on $\Omega$, defined for all random variable $Z$ by: $\left \| Z \right \|_{\mathbb{L}^{p}} = \left[ \mathbb{E} \left( |Z|^{p} \right) \right]^{\frac{1}{p}}$ . We say that $Z$ is in $\mathbb{L}^{p}(\Omega)$ if $\left[ \mathbb{E} \left( |Z|^{p} \right) \right]^{\frac{1}{p}} < \infty$.

The error process $(\epsilon_{i})_{i \in \mathbb{Z}}$ is assumed to be strictly stationary with zero mean. Moreover, for all $i$ in $\mathbb{Z}$, $\epsilon_{i}$ is supposed to be in $\mathbb{L}^{2}(\Omega)$. More precisely, the error process satisfies, for all $i$ in $\mathbb{Z}$: 
\[\epsilon_{i} = \epsilon_{0} \circ \mathbb{T}^{i},\]
where $\mathbb{T}: \Omega \rightarrow \Omega$ is a bijective bimeasurable transformation  preserving the probability measure $\mathbb{P}$. Note that any strictly stationary process can be represented in this way.

Let ($\mathcal{F}_{i}$)$_{i \in \mathbb{Z}}$ be a non-decreasing filtration built as follows, for all $i$: 
\[\mathcal{F}_{i} = \mathbb{T}^{-i}(\mathcal{F}_{0}),\]
where $\mathcal{F}_{0}$ is a sub-$\sigma$-algebra of $\mathcal{F}$ such that $\mathcal{F}_{0} \subseteq \mathbb{T}^{-1}(\mathcal{F}_{0})$. For instance, one can choose the past $\sigma$-algebra before time $0$: $\mathcal{F}_{0} = \sigma(\epsilon_{k}, k \leq 0)$, and then $\mathcal{F}_{i} = \sigma(\epsilon_{k}, k \leq i)$. In that case, $\epsilon_{0}$ is $\mathcal{F}_{0}$-measurable.

As in Hannan, we shall always suppose that $\mathcal{F}_{-\infty} = \underset{i \in \mathbb{Z}}{\bigcap} \mathcal{F}_{i}$ is trivial. Moreover $\epsilon_{0}$ is assumed $\mathcal{F}_{\infty}$-measurable. These assumptions imply that the $\epsilon_{i}$'s are all regular random variables in the following sense:

\begin{defi}[Regular random variable]
Let $Z$ be a random variable in $L^{1}(\Omega)$. We say that $Z$ is regular with respect to the filtration $(\mathcal{F}_{i})_{i \in \mathbb{Z}}$ if $\mathbb{E}(Z | \mathcal{F}_{-\infty}) = \mathbb{E}(Z)$ almost surely and if $Z$ is $\mathcal{F}_{\infty}$-measurable. 
\end{defi}
This implies that there exists a spectral density $f$ for the error process, defined on $[-\pi, \pi]$. The autocovariance function $\gamma$ of the process $\epsilon$ then satisfies: 
\[\gamma(k) = \mathrm{Cov} (\epsilon_{m}, \epsilon_{m+k}) = \mathbb{E}(\epsilon_{m}\epsilon_{m+k}) = \int_{-\pi}^{\pi} e^{i k \lambda} f(\lambda) d\lambda.\]

\subsection{Hannan's Central Limit Theorem}

Let $\hat{\beta}$ be the usual least squares estimator for the unknown vector $\beta$.
Hannan \cite{hannan73clt} has shown a Central Limit Theorem for $\hat{\beta}$ when the error process is stationary. In this section, the conditions for applying this theorem are recalled.

Let $(P_{j})_{j \in \mathbb{Z}}$ be a family of projection operators, defined for all $j$ in $\mathbb{Z}$ and for any $Z$ in $\mathbb{L}^{2}(\Omega)$ by:
\[P_{j}(Z) = \mathbb{E}(Z | \mathcal{F}_{j}) - \mathbb{E}(Z | \mathcal{F}_{j-1}).\]
We shall always assume that Hannan's condition on the error process is satisfied:
\begin{equation}
\sum_{i \in \mathbb{Z}} \left \| P_{0}(\epsilon_{i}) \right \|_{\mathbb{L}^{2}} < +\infty.
\tag{C1}
\label{0}
\end{equation}
Note that this condition implies that:
\begin{equation}
\sum_{k \in \mathbb{Z}} \left| \gamma(k) \right| < \infty,
\label{0bis}
\end{equation}
(see for instance \cite{dmv2007weak}).

Hannan's condition provides a very general framework for stationary processes. The hypothesis~\eqref{0} is a sharp condition to have a Central Limit Theorem for the partial sum sequence (see the paper of Dedecker, Merlevède and Voln\'y \cite{dmv2007weak} for more details). Notice that the condition~\eqref{0bis} implies that the error process is short-range dependent. 
However, Hannan's condition is satisfied for most short-range dependent stationary processes.  In particular, it is less restrictive that the well-known condition of Gordin \cite{gordin1969central}. Moreover the property of $2$-strong stability introduced by Wu \cite{wu2005nonlinear} is more restrictive than Hannan's condition. This property of $2$-strong stability will be recalled in Section $4$, where large classes of examples will be fully described.

Let us now recall Hannan’s assumptions on the design. Let us introduce:
\begin{equation}
d_{j}(n) = \left \| X_{.,j} \right \|_{2} = \sqrt{\sum_{i=1}^{n} x_{i, j}^{2}},
\end{equation}
and let $D(n)$ be the diagonal matrix with diagonal term $d_{j}(n)$ for $j$ in $\{1, \ldots, p\}$.

Following Hannan, we also require that the columns of the design $X$ satisfy the following conditions:

\begin{equation}
\forall j \in \{1, \ldots,p\}, \qquad \lim_{n \rightarrow \infty} d_{j}(n) = \infty,
\tag{C2}
\label{1}
\end{equation}
and:
\begin{equation}
\forall j \in \{1, \ldots, p\}, \qquad \lim_{n \rightarrow \infty} \frac{\sup_{1 \leq i \leq n} \left | x_{i,j} \right |}{d_{j}(n)} = 0.
\tag{C3}
\label{2}
\end{equation}
Moreover, we assume that the following limits exist: 

\begin{equation}
\forall j, l \in \{1, \ldots, p\}, \qquad \rho_{j,l}(k) = \lim_{n \rightarrow \infty} \sum_{m=1}^{n-k} \frac{x_{m, j} x_{m+k,l}}{d_{j}(n)d_{l}(n)}.
\tag{C4}
\label{3}
\end{equation}

Notice that there is a misprint in Hannan’s paper (the supremum is missing on
condition~\eqref{2}). Note that Conditions~\eqref{1} and~\eqref{2} correspond to the usual Lindeberg’s conditions for linear statistics in the i.i.d. case. In the dependent case, we also need Condition~\eqref{3}.

The $p \times p$ matrix formed by the coefficients $\rho_{j,l}(k)$ is called $R(k)$:
\begin{equation}
R(k) = [\rho_{j,l}(k)] = \int_{-\pi}^{\pi} e^{i k \lambda} F_{X}(d\lambda),
\label{4}
\end{equation}
where $F_{X}$ is the spectral measure associated with the matrix $R(k)$. The matrix $R(0)$ is supposed to be positive definite: 
\begin{equation}
R(0) > 0.
\tag{C5}
\label{4bis}
\end{equation}

Let then $F$ and $G$ be the matrices:
\begin{equation}
F = \frac{1}{2\pi} \int_{-\pi}^{\pi} F_{X}(d\lambda),
\label{5}
\end{equation}
\begin{equation}
G = \frac{1}{2\pi} \int_{-\pi}^{\pi} F_{X}(d\lambda) \otimes f(\lambda).
\label{6}
\end{equation}

The Central Limit Theorem for the regression parameter, due to Hannan \cite{hannan73clt}, can be stated as follows:

\begin{theo}
Let $(\epsilon_{i})_{i \in \mathbb{Z}}$ be a stationary process with zero mean. Assume that $\mathcal{F}_{-\infty}$ is trivial, $\epsilon_{0}$ is $\mathcal{F}_{\infty}$-measurable, and that the sequence $(\epsilon_{i})_{i \in \mathbb{Z}}$ satisfies Hannan's condition~\eqref{0}. Assume that the design $X$ satisfies the conditions~\eqref{1},~\eqref{2},~\eqref{3} and~\eqref{4bis}. 
Then: 
\begin{equation}
D(n)(\hat{\beta} - \beta) \xrightarrow[n \rightarrow \infty]{\mathcal{L}} \mathcal{N}(0, F^{-1}GF^{-1}).
\label{7}
\end{equation}
Furthermore, there is the convergence of second order moment: \footnote{The transpose of a matrix $X$ is denoted by $X^{t}$.}
\begin{equation}
\mathbb{E} \left( D(n) (\hat{\beta} -\beta) (\hat{\beta} -\beta)^{t} D(n)^{t} \right) \xrightarrow[n \rightarrow \infty]{} F^{-1}GF^{-1}.
\label{9}
\end{equation}
\label{8}
\end{theo}

\subsection{Regular design}

Theorem~\ref{8} is very general because it includes a very large class of designs. In this paper, we will focus on the case where the design is regular in the following sense:

\begin{defi}[Regular design]
A fixed design $X$ is called regular if, for any $j, l$ in $\{1, \ldots, p\}$, the coefficients $\rho_{j,l}(k)$ do not depend on $k$.
\end{defi}

A large class of regular designs is the one for which the columns are regularly varying sequences. Let us recall the definition of regularly varying sequences:

\begin{defi}[Regularly varying sequence \cite{seneta2006regularly}]
A sequence $S(\cdot)$ is regularly varying if and only if it can be written as:
\[S(i) = i^{\alpha} L(i),\]
where $-\infty < \alpha < \infty$ and $L(\cdot)$ is a slowly varying sequence.
\end{defi}
This includes the case of polynomial regression, where the columns are of the form: $x_{i,j} = i^{j}$.

\begin{prop}
Assume that each column $X_{.,j}$ is regularly varying with parameter $\alpha_{j}$.
If the parameters $\alpha_{j}$ are all strictly greater than $-\frac{1}{2}$, then Conditions~\eqref{1}, \eqref{2} and \eqref{3} on the design are satisfied. Moreover, for all $j$ and $l$ in $\{1, \ldots, p\}$, the coefficients $\rho_{j,l}(k)$ do not depend on $k$ and are equal to $\frac{\sqrt{2\alpha_{j}+1} \sqrt{2\alpha_{l}+1}}{\alpha_{j}+\alpha_{l}+1}$. Thereby, the design is regular, and~\eqref{4bis} is satisfied provided $\alpha_{j} \neq \alpha_{l}$ for any distinct $j, l$ in $\{1, \ldots, p\}$.
\label{9ajout}
\end{prop}

Another important class of regular designs are the ANOVA type designs. An ANOVA design is represented by a matrix whose column vectors are mutually orthogonal. Each coordinate of a column is either $0$ or $1$, with consecutive sequences of $1$'s. The number of $0$'s and $1$'s in each column tends to infinity as $n$ tends to infinity.

Note that a design whose columns are either ANOVA or regularly varying is again a regular design. 

\subsection{The asymptotic covariance matrix for regular design}

For regular design, the asymptotic covariance matrix is easy to compute. Actually, we shall see that it is the same as in the case where the errors are independent up to a multiplicative factor. 
 
More precisely, the usual variance term $\sigma^{2} = \mathbb{E}(\epsilon_{0}^{2})$ should be replaced by the sum of covariances: $\sum_{k} \gamma(k)$.

Since the coefficients $\rho_{j,l}(k)$ are constant, the spectral measure $F_{X}$ is the product of a Dirac mass at $0$, denoted $\delta_{0}$, with the matrix $R(k)$; consequently the spectral measure $F_{X}$ is equal to $\delta_{0}R(0)$. Notice that, in the case of regular design, the matrix $R(k) = [\rho_{j,l}(k)]$ is equal to $R(0) = [\rho_{j,l}(0)]$.

Thereby the matrix $F$ and $G$ can be computed explicitly:
\begin{equation}
F = \frac{1}{2\pi} \int_{-\pi}^{\pi} F_{X}(d\lambda) = \frac{1}{2\pi} \int_{-\pi}^{\pi} R(0) \delta_{0}(d\lambda) = \frac{1}{2\pi} R(0),
\label{10}
\end{equation}
\begin{equation}
G = \frac{1}{2\pi} \int_{-\pi}^{\pi} F_{X}(d\lambda) \otimes f(\lambda) = \frac{1}{2\pi} \int_{-\pi}^{\pi} R(0) \otimes f(\lambda) \delta_{0}(d\lambda) = \frac{1}{2\pi} R(0) \otimes f(0) = f(0)F.
\label{11}
\end{equation}
Thus, using~\eqref{10} and~\eqref{11}, the covariance matrix can be written as:
\[F^{-1}GF^{-1} = f(0)F^{-1}.\]

The connection between the spectral density and the autocovariance function is known:
\[f(\lambda) = \frac{1}{2\pi} \sum_{k=-\infty}^{\infty} \gamma(k) e^{- i k \lambda}, \qquad \lambda \in [-\pi, \pi].\]
and at the point $0$:
\[f(0) =\frac{1}{2\pi} \sum_{k=-\infty}^{\infty} \gamma(k).\]

Thereby the covariance matrix can be written:
\[f(0)  F^{-1} = \left( \frac{1}{2\pi} \sum_{k=-\infty}^{\infty} \gamma(k) \right) F^{-1} = \left( \sum_{k=-\infty}^{\infty} \gamma(k) \right) R(0)^{-1},\]
since $F = \frac{R(0)}{2 \pi}$ and $F^{-1} = 2 \pi R(0)^{-1}$.\\

In conclusion, for regular design the following corollary holds:

\begin{Cor}
Under the assumptions of Theorem~\ref{8}, if moreover the design $X$ is regular, then:
\begin{equation}
D(n)(\hat{\beta} - \beta) \xrightarrow[n \rightarrow \infty]{\mathcal{L}} \mathcal{N} \left( 0, \left( \sum_{k=-\infty}^{\infty} \gamma(k) \right) R(0)^{-1} \right),
\label{15}
\end{equation}
and we have the convergence of the second order moment:
\begin{equation}
\mathbb{E} \left( D(n) (\hat{\beta} - \beta) (\hat{\beta} - \beta)^{t} D(n)^{t} \right) \xrightarrow[n \rightarrow \infty]{} \left( \sum_{k=-\infty}^{\infty} \gamma(k) \right) R(0)^{-1}.
\label{15bis}
\end{equation}
\label{15ter}
\end{Cor}

One can see that, in the case of regular design, the asymptotic covariance matrix is similar to the one in the case where the random variables ($\epsilon_{i}$) are i.i.d.; the variance term $\sigma^{2}$ is replaced by the series of covariances.
Actually the matrix $R(0)^{-1}$ is the normalised limit of the matrix $(X^{t}X)^{-1}$. It is formed by the coefficients $\rho_{j,l}(0)$, which are, in this case, the limit of the normalised scalar products between the columns of the design. 

Thus, to obtain confidence regions and tests for $\beta$, an estimator of the covariance matrix is needed. More precisely, it is necessary to estimate the quantity:
\begin{equation}
\sum_{k=-\infty}^{\infty} \gamma(k).
\label{16}
\end{equation}

\section{Estimation of the series of covariances}

The properties of spectral density estimates have been discussed in many classical textbooks on time series; see, for instance,  Anderson \cite{anderson2011statistical}, Brillinger \cite{brillinger2001time}, Brockwell and Davis \cite{brockwell2013time}, Grenander and Rosenblatt \cite{grenander2008statistical}, Priestley \cite{priestley1981spectral} and Rosenblatt \cite{rosenblatt2012stationary} among others. But many of the previous results require restrictive conditions on the underlying processes (linear structure or strong mixing conditions). Wu \cite{wuspectraldensity} has developed an asymptotic theory for the spectral density estimate $f_{n}(\lambda)$, defined at~\eqref{17}, which extends the applicability of spectral analysis to nonlinear and/or non-strong mixing processes. In particular, he also proved a Central Limit Theorem and deviation inequalities for $f_{n}(\lambda)$. However, to show his results, Wu uses a notion of dependence that is more restrictive than Hannan's.

In this section, we propose an estimator of the spectral density under Hannan's dependence condition. Here, contrary to the precise results of Wu (Central Limit Theorem, deviation inequalities), we shall only focus on the consistency of the estimator. 

Let us first consider a preliminary random function defined as follows, for $\lambda$ in $[-\pi,\pi]$:
\begin{equation}
f_{n}(\lambda) = \frac{1}{2\pi} \sum_{|k| \leq n-1} K \left( \frac{|k|}{c_{n}} \right) \hat{\gamma}_{k} e^{i k \lambda},
\label{17}
\end{equation}
where:
\begin{equation}
\hat{\gamma}_{k} = \frac{1}{n} \sum_{j=1}^{n-|k|} \epsilon_{j} \epsilon_{j+|k|}, \qquad 0 \leq |k| \leq (n-1),
\label{18}
\end{equation}
and $K$ is the kernel defined by:
\[
\left\{
\begin{array}{r c l}
K(x) &=& 1 \phantom{- |x| 1 1} \quad  if\  |x| < 1\\
K(x) &=& 2 - |x| \phantom{1\quad} if\  1 \leq |x| \leq 2\\
K(x) &=& 0 \phantom{- |x| 1 1}  \quad if\  |x| > 2.\\
\end{array}
\right.
\]
The sequence of positive real numbers $c_{n}$ is such that $c_{n}$ tends to infinity and $\frac{c_{n}}{n}$ tends to $0$ when $n$ tends to infinity.

In our context, $(\epsilon_{i})_{i \in \{1, \ldots, n\}}$ is not observed. Only the residuals are available:
\[\hat{\epsilon}_{i} = Y_{i} - (x_{i})^{t} \hat{\beta} = Y_{i} - \sum_{j=1}^{p} x_{i,j} \hat{\beta}_{j},\]
because only the data $Y$ and the design $X$ are observed. Consequently, we consider the following estimator:
\begin{equation}
f_{n}^{\ast}(\lambda) = \frac{1}{2\pi} \sum_{|k| \leq n-1} K \left( \frac{|k|}{c_{n}} \right) \hat{\gamma}_{k}^{\ast} e^{i k \lambda}, \qquad \lambda \in [-\pi,\pi],
\label{19}
\end{equation}
where:
\[\hat{\gamma}_{k}^{\ast} = \frac{1}{n} \sum_{j=1}^{n-|k|} \hat{\epsilon}_{j} \hat{\epsilon}_{j+|k|}, \qquad 0 \leq |k| \leq (n-1).\]

\begin{theo}
Let $c_{n}$ be a sequence of positive real numbers such that $c_{n}$ tends to infinity as $n$ tends to infinity, and:
\begin{equation}
c_{n} \mathbb{E} \left( \left| \epsilon_{0} \right|^{2} \left( 1 \wedge \frac{c_{n}}{n} \left| \epsilon_{0} \right|^{2} \right) \right) \xrightarrow[n \rightarrow \infty]{} 0.
\label{48}
\end{equation}
Then, under the assumptions of Theorem~\ref{8}:
\begin{equation}
\sup_{\lambda \in [-\pi,\pi]} \left \| f_{n}^{\ast}(\lambda) - f(\lambda) \right \|_{\mathbb{L}^{1}} \xrightarrow[n \rightarrow \infty]{} 0.
\label{49}
\end{equation}
\label{50}
\end{theo}

\begin{Rem}
If $\epsilon_{0}$ is in $\mathbb{L}^{2}$, then there exists $c_{n} \rightarrow \infty$ such that~\eqref{48} holds.
\end{Rem}
\begin{Rem}
Let us suppose that the random variable $\epsilon_{0}$ is such that $\mathbb{E} \left( \left | \epsilon_{0} \right |^{\delta+2} \right) < \infty$, with $\delta \in ]0,2]$. Since for all real $x$, $1 \wedge |x|^{2} \leq |x|^{\delta}$, we have:
\[c_{n} \mathbb{E} \left( \left| \epsilon_{0} \right|^{2} \left( 1 \wedge \frac{c_{n}}{n} \left| \epsilon_{0} \right|^{2} \right) \right)  \leq c_{n} \mathbb{E} \left( \left| \epsilon_{0} \right|^{2} \frac{c_{n}^{\delta/2}}{n^{\delta/2}} |\epsilon_{0}|^{\delta} \right) \leq \frac{c_{n}^{1+\delta/2}}{n^{\delta/2}} \mathbb{E} \left( \left | \epsilon_{0} \right |^{\delta+2} \right).\] 
Thus if $c_{n}$ satisfies $\frac{c_{n}^{1+\delta/2}}{n^{\delta/2}} \xrightarrow[n \rightarrow \infty]{} 0$, then~\eqref{48} holds.
In particular, if the random variable $\epsilon_{0}$ has a fourth order moment, then the condition on $c_{n}$ is $\frac{c_{n}^{2}}{n} \xrightarrow[n \rightarrow \infty]{} 0$.
\end{Rem}

Theorem~\ref{8} implies the following result:

\begin{Cor}
Under the assumptions of Corollary \ref{15ter}, and if $f(0) > 0$, then:
\begin{equation}
\frac{R(0)^{\frac{1}{2}}}{\sqrt{2\pi f_{n}^{\ast}(0) }} D(n)(\hat{\beta} - \beta) \xrightarrow[n \rightarrow \infty]{\mathcal{L}} \mathcal{N}(0,I_{p}),
\label{51}
\end{equation}
where $I_{p}$ is the $p \times p$ identity matrix.
\label{52}
\end{Cor}

\begin{Rem}
Let us emphasize that Corollary~\ref{52} can be used to compute confidence regions for the parameter $\beta$ or to modify the usual Fisher tests on the linear regression model in the short-range dependent case. Some simulations show that the procedure works well for a large variety of examples, including non-mixing processes in the sense of Rosenblatt \cite{rosenblatt1956central}, processes that are not $2$-stable in the sense of Wu (see Section $4.2$), or dynamical systems with slowly decaying correlations (such as the Liverani, Saussol and Vaienti maps described in Section $4.4$, for $\gamma < 1/2$). A practical choice of $c_{n}$ can be done by examining the graph of the empirical autocovariance of the residuals.
\end{Rem}

\section{Examples of stationary processes}

In this section, we present some classes of stationary processes satisfying Hannan's condition.

\subsection{Functions of Linear processes}

A large class of stationary processes for which one can check Hannan's condition is the class of smooth functions of linear processes generated by i.i.d. random variables.

Let us take $\Omega = \mathbb{R}^{\mathbb{Z}}$ and $\mathbb{P} = \mu^{\otimes \mathbb{Z}}$, where $\mu$ is a probability measure on $\mathbb{R}$. Let ($\eta_{i}, i \in \mathbb{Z}$) be a sequence of i.i.d. random variables with marginal distribution $\mu$. Let $(a_{i})_{i \in \mathbb{Z}}$ be a sequence of real numbers in $l^{1}$, and assume that $\sum_{i \in \mathbb{Z}} a_{i} \eta_{i}$ is defined almost surely. The random variable $\epsilon_{0}$ is square integrable and is regular with respect to the $\sigma$-algebras : $\mathcal{F}_{i} = \sigma (\eta_{j}, j \leq i)$. We focus on functions of real-valued linear processes:
\[\epsilon_{k} = f \left( \sum_{i \in \mathbb{Z}} a_{i} \eta_{k-i} \right) - \mathbb{E} \left( f \left( \sum_{i \in \mathbb{Z}} a_{i} \eta_{k-i} \right) \right).\]

Let us define the modulus of continuity of $f$ on the interval $[-M, M]$ by:
\[\omega_{\infty,f}(h,M) = \sup_{|t| \leq h, |x| \leq M, |x+t| \leq M} \left | f(x+t) -f(x) \right | .\]
Let $(\eta_{i}')_{i \in \mathbb{Z}}$ be an independent copy of $(\eta_{i})_{i \in \mathbb{Z}}$, and let:
\[M_{k} = \max \left\{ \left | \sum_{i \in \mathbb{Z}} a_{i} \eta_{i}' \right |, \left | a_{k} \eta_{0} + \sum_{i \neq k} a_{i} \eta_{i}' \right | \right\}.\]

According to Section $5$ in the paper of Dedecker, Merlevède, Voln\'y \cite{dmv2007weak}, if the following condition holds:
\begin{equation}
\sum_{k \in \mathbb{Z}} \Big \| \omega_{\infty,f}(|a_{k}| |\eta_{0}|, M_{k}) \wedge \left \| \epsilon_{0} \right \|_{\infty} \Big \|_{\mathbb{L}^{2}} < \infty,
\label{80}
\end{equation}
then Hannan's condition holds.
We have an interesting application if the function $f$ is $\gamma$-Hölder on any compact set; if $\omega_{\infty,f}(h,M) \leq C h^{\gamma} M^{\alpha}$ for some $C > 0$, $\gamma \in ]0,1]$ and $\alpha \geq 0$, then~\eqref{80} holds as soon as $\sum |a_{k}|^{\gamma} < \infty$ and $\mathbb{E}(|\eta_{0}|^{2(\alpha + \gamma)}) < \infty$.

\subsection{$2$-strong stability}

Let us recall in this section the framework used by Wu. We consider stationary processes of the form:
\[\epsilon_{i} = H(\ldots, \eta_{i-1}, \eta_{i}),\]
where $\eta_{i}$, $i$ in $\mathbb{Z}$, are i.i.d. random variables and $H$ is a measurable function. 
Assume that $\epsilon_{0}$ belongs to $\mathbb{L}^{2}$, and let $\eta'_{0}$ be distributed as $\eta_{0}$ and independent of $(\eta_{i})$. Let us define the physical dependence measure in $\mathbb{L}^{2}$ \cite{wudependence}, for $j \geq 0$:
\[\delta_{2}(j) = \left \| \epsilon_{j} - \epsilon_{j}^{\ast} \right \|_{\mathbb{L}^{2}},\]
where $\epsilon_{j}^{\ast}$ is a coupled version of $\epsilon_{j}$ with $\eta_{0}$ in the latter being replaced by $\eta'_{0}$:
\[\epsilon_{j}^{\ast} = H(\ldots, \eta_{-1}, \eta'_{0}, \eta_{1}, \ldots, \eta_{j-1}, \eta_{j}).\]
The sequence $(\epsilon_{i})_{i \in {\mathbb Z}}$ is said to be $2$-strong stable if:
\[\Delta_{2} = \sum_{j=0}^{\infty} \delta_{2}(j) < \infty.\]

As a consequence of Theorem $1$, $(i)-(ii)$ of Wu \cite{wu2005nonlinear}, we infer that if
$(\epsilon_{i})_{i \in {\mathbb Z}}$ is $2$-strong stable, then it satisfies Hannan's condition with respect to the filtration $\mathcal{F}_{i} = \sigma(\eta_{j}, j \leq i)$.
Many examples of $2$-strong stable processes are presented in the paper by Wu \cite{wu2005nonlinear}. We also refer to \cite{wudependence} for other examples.

\subsection{Conditions in the style of Gordin}

According to Proposition $5$ of Dedecker, Merlevède, Voln\'y \cite{dmv2007weak}, Hannan's condition holds if the error process satisfies the two following conditions:
\begin{eqnarray}
\sum_{k=1}^{\infty} \frac{1}{\sqrt{k}} \left \| \mathbb{E}(\epsilon_{k} | \mathcal{F}_{0}) \right \|_{\mathbb{L}^{2}} < \infty \label{81} \\
\sum_{k=1}^{\infty} \frac{1}{\sqrt{k}} \left \| \epsilon_{-k} - \mathbb{E}(\epsilon_{-k} | \mathcal{F}_{0}) \right \|_{\mathbb{L}^{2}} < \infty. \label{82}
\end{eqnarray}

These conditions are weaker than the well-known conditions of Gordin \cite{gordin1969central}, under which a martingale + coboundary decomposition holds in $\mathbb{L}^{2}$.
An application is given in the next subsection.

\subsection{Weak dependent coefficients}

Hannan's condition holds if the error process is weakly dependent. In this case, the $(\epsilon_{i})_{i \in \mathbb{Z}}$ process is $\mathcal{F}$-adapted and Condition~\eqref{82} is always true.

Let us recall the definitions of weak dependence coefficients, introduced by Dedecker and Prieur~\cite{dedecker_prieur}; for all integer $k \geq 0$:

\[\tilde{\phi}(k) = \tilde{\phi}(\mathcal{F}_{0}, \epsilon_{k}) = \sup_{t \in \mathbb{R}} \left \| \mathbb{P}(\epsilon_{k} \leq t | \mathcal{F}_{0}) - \mathbb{P}(\epsilon_{k} \leq t) \right \|_{\infty},\]
and:
\[\tilde{\alpha}(k) = \tilde{\alpha}(\mathcal{F}_{0}, \epsilon_{k}) = \sup_{t \in \mathbb{R}} \left \| \mathbb{P}(\epsilon_{k} \leq t | \mathcal{F}_{0}) - \mathbb{P}(\epsilon_{k} \leq t) \right \|_{\mathbb{L}^{1}}.\]

If $(\epsilon_{i})_{i \in \mathbb{Z}}$ is $\tilde{\phi}$-dependent and is in $\mathbb{L}^{p}$ with $p \in [2, +\infty[$, then by Hölder's inequality:
\[\left \| \mathbb{E}(\epsilon_{k} | \mathcal{F}_{0}) \right \|_{\mathbb{L}^{2}} \leq \left \| \mathbb{E}(\epsilon_{k} | \mathcal{F}_{0}) \right \|_{\mathbb{L}^{p}} \leq \sup_{Z \in B_{\frac{p}{p-1}}(\mathcal{F}_{0})} \mathbb{E}(Z \epsilon_{k}) \leq 2 \tilde{\phi}(k)^{\frac{p-1}{p}} \left \| \epsilon_{0} \right \|_{\mathbb{L}^{p}},\]
where for all $q \in ]1,2]$, $B_{q}(\mathcal{F}_{0})$ is the set of random variables Z, $\mathcal{F}_{0}$-measurable such that $\left \| Z \right \|_{\mathbb{L}^{q}} \leq 1$.

Consequently, if:
\begin{equation}
\sum_{k=1}^{\infty} \frac{1}{\sqrt{k}} \tilde{\phi}(k)^{\frac{p-1}{p}} < \infty,
\label{phitildedep}
\end{equation}
then the condition~\eqref{81} holds, and Hannan's condition is satisfied.\\

Now we look at the $\tilde{\alpha}$-weakly dependent sequence. 
We denote $Q_{\epsilon}$ the generalized inverse function of $x \rightarrow \mathbb{P}(|\epsilon| > x)$. If $(\epsilon_{i})_{i \in \mathbb{Z}}$ is $\tilde{\alpha}$-mixing and verifies that there exists $r \in ]2, +\infty[$, such that $\mathbb{P}(|\epsilon| \geq t) \leq t^{-r}$, then, by Cauchy-Schwarz's inequality and Rio's inequality (Theorem $1.1$ \cite{rio1999theorie}), we get:
\[\left \| \mathbb{E}(\epsilon_{k} | \mathcal{F}_{0}) \right \|_{\mathbb{L}^{2}} = \sup_{Z \in B_{2}(\mathcal{F}_{0})} \mathbb{E}(Z \epsilon_{k}) \leq 2 \left( \int_{0}^{\tilde{\alpha}(k)} Q_{\epsilon_{k}}^{2}(u) du \right)^{\frac{1}{2}}.\]
But:
\[\int_{0}^{\tilde{\alpha}(k)} Q_{\epsilon_{k}}^{2}(u) du \leq \int_{0}^{\tilde{\alpha}(k)} \frac{1}{u^{\frac{2}{r}}} du \leq \tilde{\alpha}(k)^{1-\frac{2}{r}}.\]
Hence, if:
\begin{equation}
\sum_{k=1}^{\infty} \frac{\tilde{\alpha}(k)^{\frac{1}{2} - \frac{1}{r}}}{\sqrt{k}} < \infty,
\label{16bis}
\end{equation}
then~\eqref{81} is true, and Hannan's condition is satisfied.

Notice that all we have written for $\tilde{\alpha}$-dependent sequences is also true for $\alpha$-mixing processes in the sense of Rosenblatt \cite{rosenblatt2012stationary}.\\

We now give two examples to which our results apply.
For the first example, let us consider the process ($Z_{1}, \ldots, Z_{n}$), according to the AR(1) equation: $Z_{k+1} = \frac{1}{2}(Z_{k} + \eta_{k+1})$, where $Z_{1}$ is uniformly distributed over $[0,1]$, and $(\eta_{i})_{i \geq 2}$ is a sequence of i.i.d. random variables with distribution $\mathcal{B}(1/2)$, independent of $Z_{1}$.
The transition kernel of the chain $(Z_{i})_{i \geq 1}$ is: 
\[K(g)(x) = \frac{1}{2} \left( g \left( \frac{x}{2} \right) + g \left( \frac{x+1}{2} \right) \right),\]
and the uniform distribution on $[0,1]$ is the unique invariant distribution by $K$. Hence, the chain $(Z_{i})_{i \geq 1}$ is strictly stationary.
Furthermore, it is not $\alpha$-mixing in the sense of Rosenblatt \cite{bradley1985basic}, but it is $\tilde{\phi}$-dependent. Indeed, one can prove that the coefficient $\tilde{\phi}$ of the chain $(Z_{i})_{i \geq 1}$ decreases geometrically \cite{dedecker_prieur}: $\tilde{\phi}(k) \leq 2^{-k}$. 
Let us now consider the error process built from $(Z_{i})_{i \geq 1}$:
\[\epsilon_{k} = f(Z_{k}) - \mathbb{E}(f(Z_{k})),\]
with $f$ a monotonic function from $]0,1[$ to $\mathbb{R}$, such that $\int_{0}^{1} f(x)^{2} dx < \infty$. Then the coefficients $\tilde{\phi}$ of the process $(\epsilon_{i})_{i \geq 1}$ still decrease geometrically. Consequently~\eqref{phitildedep} (and hence Hannan's condition) is satisfied for the error process $(\epsilon_{i})_{i \geq 1}$. \\

For the second example, we consider the intermittent map $\theta_{\gamma}$ from $[0,1]$ to $[0,1]$, with $\gamma$ in $]0,1[$, introduced by Liverani, Saussol and Vaienti \cite{liverani1999probabilistic}:

\[\theta_{\gamma}(x) = 
\left\{
\begin{array}{r c l}
x(1 + 2^{\gamma} x^{\gamma}) \qquad \text{if} \ x \in [0, 1/2[ \\
2x - 1 \qquad \text{if} \ x \in [1/2, 1].\\
\end{array}
\right.\]
It follows from \cite{liverani1999probabilistic} that there exists a unique absolutely continuous $\theta_{\gamma}$-invariant probability measure $\nu_{\gamma}$, with density $h_{\gamma}$.

Let us briefly describe the Markov chain associated with $\theta_{\gamma}$, and its properties. Let first $K_{\gamma}$ be the Perron-Frobenius operator of $\theta_{\gamma}$ with respect to $\nu_{\gamma}$, defined as follows: for any functions $u$, $v$ in $\mathbb{L}^{2}([0,1], \nu_{\gamma})$:
\[\nu_{\gamma}(u \cdot v \circ \theta_{\gamma}) = \nu_{\gamma}(K_{\gamma}(u) \cdot v).\]
The operator $K_{\gamma}$ is a transition kernel, and $\nu_{\gamma}$ is invariant by $K_{\gamma}$. Let now $(\xi_{i})_{i \geq 1}$ be a stationary Markov chain with transition kernel $K_{\gamma}$, and consider the error process built from $(\xi_{i})_{i \geq 1}$: $\epsilon_{k} = \xi_{k} - \mathbb{E}(\xi_{k})$. It is proved in \cite{dedecker2010some} that there exists two positive constants $A, B$ such that:
\[\frac{A}{(n+1)^{\frac{1-\gamma}{\gamma}}} \leq \tilde{\alpha}_{\epsilon}(n) \leq \frac{B}{(n+1)^{\frac{1-\gamma}{\gamma}}}.\]
Moreover, the chain $(\epsilon_{i})_{i \geq 1}$ is not $\alpha$-mixing in the sense of Rosenblatt \cite{rosenblatt1956central}.

Note that the error process $(\epsilon_{i})$ satisfies the condition~\eqref{16bis} (and hence Hannan's condition) if and only if $\gamma$ is in $]0,\frac{1}{2}[$, which corresponds to the short-range dependent case. If $\gamma$ is in $]\frac{1}{2},1[$, the error process is long-range dependent (see the Introduction of the paper~\cite{dedecker2015weak}).

\section{Proofs}

\subsection{Proposition~\ref{9ajout}}

\begin{proof}

Let us define:
\[d_{j}(n) = || X_{.,j} ||_{2} = \sqrt{\sum_{i=1}^{n} i^{2 \alpha_{j}} L(i)^{2}}.\] 
The condition~\eqref{1} is verified if:
\begin{equation}
\sum_{i=1}^{n} i^{2 \alpha_{j}} L(i)^{2} \rightarrow \infty.
\label{proof1}
\end{equation}
When $2 \alpha_{j} < -1$, it is known that~\eqref{proof1} converges. 
However, for $2 \alpha_{j} > -1$, thanks to Proposition $2.2.1$ of Pipiras and Taqqu \cite{pipiras2017long}, we have the following equivalence:
\[\sum_{i=1}^{n} i^{2 \alpha_{j}} L(i)^{2} \sim \frac{n^{2 \alpha_{j}+1} L(n)^{2}}{2 \alpha_{j}+1},\]
and this quantity diverges as $n$ tends to infinity. 
Thus the condition~\eqref{1} is satisfied if $\alpha_{j}$ is strictly greater than $-\frac{1}{2}$. We also immediately check that~\eqref{2} is satisfied. \\

Now let us compute the coefficients $\rho_{j,l}(k)$ and prove that they do not depend on $k$. For $j, l$ belonging to $\{1, \ldots, p\}$:
\[\sum_{m=1}^{n-k} \frac{x_{m,j}x_{m+k,l}}{d_{j}(n)d_{l}(n)} = \frac{\sum_{m=1}^{n-k} m^{\alpha_{j}} L(m) (m+k)^{\alpha_{l}} L'(m+k)}{\sqrt{\sum_{i=1}^{n} i^{2 \alpha_{j}} L(i)^{2}} \sqrt{\sum_{q=1}^{n} q^{2 \alpha_{l}} L'(q)^{2}}},\]
and we have:
\begin{multline}
\frac{\sum_{m=1}^{n-k} m^{\alpha_{j}} L(m) (m+k)^{\alpha_{l}} L'(m+k)}{\sqrt{\sum_{i=1}^{n} i^{2 \alpha_{j}} L(i)^{2}} \sqrt{\sum_{q=1}^{n} q^{2 \alpha_{l}} L'(q)^{2}}} \\
= \frac{\sum_{m=1}^{n-k} (m^{\alpha_{j}} ((m+k)^{\alpha_{l}} - m^{\alpha_{l}})) L(m) L'(m+k)}{\sqrt{\sum_{i=1}^{n} i^{2 \alpha_{j}} L(i)^{2}} \sqrt{\sum_{q=1}^{n} q^{2 \alpha_{l}} L'(q)^{2}}} \\
+ \frac{\sum_{m=1}^{n-k} m^{\alpha_{j}} m^{\alpha_{l}} L(m) L'(m+k)}{\sqrt{\sum_{i=1}^{n} i^{2 \alpha_{j}} L(i)^{2}} \sqrt{\sum_{q=1}^{n} q^{2 \alpha_{l}} L'(q)^{2}}}. 
\label{9bis}
\end{multline}

Let us deal with the first term of the right-hand side in~\eqref{9bis}. If $\alpha_{l} \geq 1$, we get:
\begin{multline*}
\frac{\sum_{m=1}^{n-k} (m^{\alpha_{j}} ((m+k)^{\alpha_{l}} - m^{\alpha_{l}})) L(m) L'(m+k)}{\sqrt{\sum_{i=1}^{n} i^{2 \alpha_{j}} L(i)^{2}} \sqrt{\sum_{q=1}^{n} q^{2 \alpha_{l}} L'(q)^{2}}} \\
\leq \frac{\sum_{m=1}^{n-k} (m^{\alpha_{j}} (k \alpha_{l} (m+k)^{\alpha_{l}-1})) L(m) L'(m+k)}{\sqrt{\sum_{i=1}^{n} i^{2 \alpha_{j}} L(i)^{2}} \sqrt{\sum_{q=1}^{n} q^{2 \alpha_{l}} L'(q)^{2}}} \\
\leq \frac{(k \alpha_{l})\sum_{m=1}^{n-k} m^{\alpha_{j}} (m(1+ \frac{k}{m}))^{\alpha_{l}-1} L(m) L'(m+k)}{\sqrt{\sum_{i=1}^{n} i^{2 \alpha_{j}} L(i)^{2}} \sqrt{\sum_{q=1}^{n} q^{2 \alpha_{l}} L'(q)^{2}}},
\end{multline*}
and because $\frac{k}{m}$ is smaller or equal to $k$:
\begin{multline*}
\frac{(k \alpha_{l})\sum_{m=1}^{n-k} m^{\alpha_{j}} (m(1+ \frac{k}{m}))^{\alpha_{l}-1} L(m) L'(m+k)}{\sqrt{\sum_{i=1}^{n} i^{2 \alpha_{j}} L(i)^{2}} \sqrt{\sum_{q=1}^{n} q^{2 \alpha_{l}} L'(q)^{2}}} \\
\leq \frac{(k \alpha_{l})\sum_{m=1}^{n-k} m^{\alpha_{j}} m^{\alpha_{l}-1} (1+ k)^{\alpha_{l}-1} L(m) L'(m+k)}{\sqrt{\sum_{i=1}^{n} i^{2 \alpha_{j}} L(i)^{2}} \sqrt{\sum_{q=1}^{n} q^{2 \alpha_{l}} L'(q)^{2}}} \\
\leq \frac{(k \alpha_{l}) (1+ k)^{\alpha_{l}-1} \sum_{m=1}^{n} m^{\alpha_{j} + \alpha_{l} - 1} L(m) L'(m+k)}{\sqrt{\sum_{i=1}^{n} i^{2\alpha_{j}} L(i)^{2}} \sqrt{\sum_{q=1}^{n} q^{2 \alpha_{l}} L'(q)^{2}}}.
\end{multline*}
Using again the proposition of Pipiras and Taqqu, we have:
\begin{multline*}
\frac{(k \alpha_{l}) (1+ k)^{\alpha_{l}-1} \sum_{m=1}^{n} m^{\alpha_{j} + \alpha_{l} - 1} L(m) L'(m+k)}{\sqrt{\sum_{i=1}^{n} i^{2 \alpha_{j}} L(i)^{2}} \sqrt{\sum_{q=1}^{n} q^{2 \alpha_{l}} L'(q)^{2}}} \\
\sim \frac{(k \alpha_{l}) (1+ k)^{\alpha_{l}-1} \frac{n^{\alpha_{j} + \alpha_{l}}}{\alpha_{j}+\alpha_{l}} L(n) L'(n+k)}{\sqrt{\frac{n^{2 \alpha_{j}+1}}{2 \alpha_{j}+1} L(n)^{2}} \sqrt{\frac{n^{2 \alpha_{l}+1}}{2 \alpha_{l}+1} L'(n)^{2}}} \\
\sim \frac{\sqrt{2 \alpha_{j}+1} \sqrt{2 \alpha_{l}+1} (k \alpha_{l}) (1+ k)^{\alpha_{l}-1}}{\alpha_{j}+\alpha_{l}} \frac{1}{n} \frac{L'(n+k)}{L'(n)},
\end{multline*}
and this quantity tends to $0$ as $n$ tends to infinity.

With the same idea, if $0 < \alpha_{l} < 1$ and again for the first term on the right-hand side in~\eqref{9bis}, we have:
\begin{multline*}
\frac{\sum_{m=1}^{n-k} m^{\alpha_{j}} ((m+k)^{\alpha_{l}} - m^{\alpha_{l}}) L(m) L'(m+k)}{\sqrt{\sum_{i=1}^{n} i^{2 \alpha_{j}} L(i)^{2}} \sqrt{\sum_{q=1}^{n} q^{2 \alpha_{l}} L'(q)^{2}}} \\
\leq \frac{\sum_{m=1}^{n-k} (m^{\alpha_{j}} (k \alpha_{l} m^{\alpha_{l}-1})) L(m) L'(m+k)}{\sqrt{\sum_{i=1}^{n} i^{2 \alpha_{j}} L(i)^{2}} \sqrt{\sum_{q=1}^{n} q^{2 \alpha_{l}} L'(q)^{2}}} \\
\leq \frac{(k \alpha_{l}) \sum_{m=1}^{n} m^{\alpha_{j} + \alpha_{l} - 1} L(m) L'(m+k)}{\sqrt{\sum_{i=1}^{n} i^{2 \alpha_{j}} L(i)^{2}} \sqrt{\sum_{q=1}^{n} q^{2 \alpha_{l}} L'(q)^{2}}} .
\end{multline*}
If $\alpha_{j}+\alpha_{l} > 0$, we can use the equivalence of Pipiras and Taqqu and show that it converges to $0$:
\begin{eqnarray*}
\frac{(k \alpha_{l}) \sum_{m=1}^{n} m^{\alpha_{j} + \alpha_{l} - 1} L(m) L'(m+k)}{\sqrt{\sum_{i=1}^{n} i^{2 \alpha_{j}} L(i)^{2}} \sqrt{\sum_{q=1}^{n} q^{2 \alpha_{l}} L'(q)^{2}}} 
&\sim & \frac{(k \alpha_{l}) \sqrt{2 \alpha_{j}+1} \sqrt{2\alpha_{l}+1}}{\alpha_{j}+\alpha_{l}} \frac{1}{n} \frac{L'(n+k)}{L'(n)}.
\end{eqnarray*}
If $\alpha_{j}+\alpha_{l} < 0$, the quantity converges to $0$, because the numerator is summable and the denominator tends to infinity. Furthermore, if $\alpha_{j}+\alpha_{l} = 0$, the quantity converges to $0$ too.

Finally, if $-\frac{1}{2} < \alpha_{l} < 0$, we have:
\begin{eqnarray*}
&\phantom{=}&
\frac{\sum_{m=1}^{n-k} (m^{\alpha_{j}} ((m+k)^{\alpha_{l}} - m^{\alpha_{l}})) L(m) L'(m+k)}{\sqrt{\sum_{i=1}^{n} i^{2 \alpha_{j}} L(i)^{2}} \sqrt{\sum_{q=1}^{n} q^{2 \alpha_{l}} L'(q)^{2}}} \\
&\leq & \frac{\sum_{m=1}^{n-k} (m^{\alpha_{j}} \left| (m+k)^{\alpha_{l}} - m^{\alpha_{l}} \right| ) L(m) L'(m+k)}{\sqrt{\sum_{i=1}^{n} i^{2 \alpha_{j}} L(i)^{2}} \sqrt{\sum_{q=1}^{n} q^{2 \alpha_{l}} L'(q)^{2}}} \\
&\leq & \frac{\sum_{m=1}^{n-k} (m^{\alpha_{j}} (k | \alpha_{l} | m^{\alpha_{l}-1})) L(m) L'(m+k)}{\sqrt{\sum_{i=1}^{n} i^{2 \alpha_{j}} L(i)^{2}} \sqrt{\sum_{q=1}^{n} q^{2 \alpha_{l}} L'(q)^{2}}} \\
&\leq & \frac{(k | \alpha_{l} |) \sum_{m=1}^{n} m^{\alpha_{j}+\alpha_{l}-1} L(m) L'(m+k)}{\sqrt{\sum_{i=1}^{n} i^{2 \alpha_{j}} L(i)^{2}} \sqrt{\sum_{q=1}^{n} q^{2 \alpha_{l}} L'(q)^{2}}},
\end{eqnarray*}
and we get the same results as above.\\

For the second term on the right-hand side in~\eqref{9bis}, we use again the proposition of Pipiras and Taqqu:
\begin{eqnarray*}
\frac{\sum_{m=1}^{n-k} m^{\alpha_{j}+\alpha_{l}} L(m) L'(m+k)}{\sqrt{\sum_{i=1}^{n} i^{2 \alpha_{j}} L(i)^{2}} \sqrt{\sum_{q=1}^{n} q^{2 \alpha_{l}} L'(q)^{2}}}
&\sim & \frac{\frac{(n-k)^{\alpha_{j}+\alpha_{l}+1}}{\alpha_{j}+\alpha_{l}+1} L(n-k) L'(n)}{\sqrt{\frac{n^{2\alpha_{j}+1}}{2 \alpha_{j}+1} L(n)^{2}} \sqrt{\frac{n^{2 \alpha_{l}+1}}{2 \alpha_{l}+1} L'(n)^{2}}} \\
&\sim & \frac{\sqrt{2 \alpha_{j}+1} \sqrt{2\alpha_{l}+1}}{\alpha_{j}+\alpha_{l}+1} \frac{(n-k)^{\alpha_{j}+\alpha_{l}+1}}{n^{\alpha_{j}+1/2} n^{\alpha_{l}+1/2}} \frac{L(n-k)}{L(n)},
\end{eqnarray*}
and this quantity converges to $\frac{\sqrt{2\alpha_{j}+1} \sqrt{2 \alpha_{l}+1}}{\alpha_{j}+\alpha_{l}+1}$. 

Thereby the coefficients $\rho_{j,l}(k)$ are constants and equal to $\frac{\sqrt{2\alpha_{j}+1} \sqrt{2\alpha_{l}+1}}{\alpha_{j}+\alpha_{l}+1}$.

\end{proof}

\subsection{Theorem~\ref{50}}

\begin{proof}
The proof of Theorem~\ref{50} is split in two parts. Indeed, notice that:
\[\left \| f_{n}^{\ast}(\lambda) - f(\lambda) \right \|_{\mathbb{L}^{1}} \leq \left \| f_{n}^{\ast}(\lambda) - f_{n}(\lambda) \right \|_{\mathbb{L}^{1}} + \left \| f_{n}(\lambda) - f(\lambda) \right \|_{\mathbb{L}^{1}}.\]
The proof is complete with Propositions~\ref{200bis} and~\ref{201bis}:

\begin{prop}
Under the assumptions of Theorem~\ref{50}, we have:
\begin{equation}
\lim_{n \rightarrow \infty} \sup_{\lambda \in [-\pi,\pi]} \left \| f_{n}(\lambda) - f(\lambda) \right \|_{\mathbb{L}^{1}} = 0.
\label{200}
\end{equation}
\label{200bis}
\end{prop}

\begin{prop}
Under the assumptions of Theorem~\ref{50}, we have:
\begin{equation}
\lim_{n \rightarrow \infty} \sup_{\lambda \in [-\pi,\pi]} \left \| f_{n}^{\ast}(\lambda) - f_{n}(\lambda) \right \|_{\mathbb{L}^{1}} = 0.
\label{201}
\end{equation}
\label{201bis}
\end{prop}

\end{proof}

\subsubsection{Proposition~\ref{200bis}}

\begin{proof}
Without loss of generality, $c_{n}$ is chosen such that $2c_{n} \leq n-1$.
Let $m$ be an integer such that: $1 \leq 2m \leq 2c_{n} \leq n-1$. For all $i \in \mathbb{Z}$, define: 
\begin{equation}
\tilde{\epsilon}_{i,m} = \mathbb{E}(\epsilon_{i} | \mathcal{F}_{i+m}) - \mathbb{E}(\epsilon_{i} | \mathcal{F}_{i-m}),
\label{202}
\end{equation}
and notice that $\mathbb{E}(\tilde{\epsilon}_{i,m}) = 0$. The associated spectral density estimate is defined as follows:
\[\tilde{f}_{n}^{m}(\lambda) = \frac{1}{2\pi} \sum_{|k| \leq n-1} K \left(\frac{|k|}{c_{n}} \right) \hat{\tilde{\gamma}}_{k,m} e^{i k \lambda}\text{,} \ \quad \lambda \in [-\pi,\pi],\]
where : 
\[\hat{\tilde{\gamma}}_{k,m} = \frac{1}{n} \sum_{j=1}^{n-|k|} \tilde{\epsilon}_{j,m} \tilde{\epsilon}_{j+|k|,m}, \ \quad |k| \leq n-1.\]

By the triangle inequality, it follows that:
\begin{eqnarray*}
\left \| f_{n} \left( \lambda \right) - f \left( \lambda \right) \right \|_{\mathbb{L}^{1}} 
&\leq & \left \| f_{n}(\lambda) - \tilde{f}_{n}^{m}(\lambda) \right \|_{\mathbb{L}^{1}} + \left \| \tilde{f}_{n}^{m}(\lambda) - \mathbb{E}(\tilde{f}_{n}^{m}(\lambda)) \right \|_{\mathbb{L}^{1}} \\
&& +\: \left | \mathbb{E}(\tilde{f}_{n}^{m}(\lambda)) - \mathbb{E}(f_{n}(\lambda)) \right | + \left \| \mathbb{E}(f_{n}(\lambda)) - f(\lambda) \right \|_{\mathbb{L}^{1}} \\
&\leq & 2\left \| \tilde{f}_{n}^{m}(\lambda) - f_{n}(\lambda) \right \|_{\mathbb{L}^{1}} + \left \| \tilde{f}_{n}^{m}(\lambda) - \mathbb{E}(\tilde{f}_{n}^{m}(\lambda)) \right \|_{\mathbb{L}^{1}} + \left \| \mathbb{E}(f_{n}(\lambda)) - f(\lambda) \right \|_{\mathbb{L}^{1}},
\end{eqnarray*}
because $\left | \mathbb{E}(\tilde{f}_{n}^{m}(\lambda)) - \mathbb{E}(f_{n}(\lambda)) \right | \leq \left \| \tilde{f}_{n}^{m}(\lambda) - f_{n}(\lambda) \right \|_{\mathbb{L}^{1}}$.\\

The proof is complete using Lemmas~\ref{205bis},~\ref{206bis} and~\ref{207bis}:

\begin{lem}
Under the assumptions of Theorem~\ref{50}, we have:
\begin{equation}
\lim_{n \rightarrow \infty} \sup_{\lambda \in [-\pi,\pi]} \left \| \mathbb{E}(f_{n}(\lambda)) - f(\lambda) \right \|_{\mathbb{L}^{1}} = 0.
\label{205}
\end{equation}
\label{205bis}
\end{lem}

\begin{lem}
Under the assumptions of Theorem~\ref{50}, we have:
\begin{equation}
\lim_{m \rightarrow \infty} \limsup_{n \rightarrow \infty} \sup_{\lambda \in [-\pi,\pi]} \left \| \tilde{f}_{n}^{m}(\lambda) - f_{n}(\lambda) \right \|_{\mathbb{L}^{1}} = 0.
\label{206}
\end{equation}
\label{206bis}
\end{lem}

\begin{lem}
Under the assumptions of Theorem~\ref{50}, we have:
\begin{equation}
\lim_{m \rightarrow \infty} \limsup_{n \rightarrow \infty} \sup_{\lambda \in [-\pi,\pi]} \left \| \tilde{f}_{n}^{m}(\lambda) - \mathbb{E}(\tilde{f}_{n}^{m}(\lambda)) \right \|_{\mathbb{L}^{1}} = 0.
\label{207}
\end{equation}
\label{207bis}
\end{lem}

\end{proof}

\begin{proof}[\textbf{Proof of Lemma~\ref{205bis}}]

By the properties of expectation and by stationarity:
\[\mathbb{E} \left( f_{n}(\lambda) \right) = \frac{1}{2\pi} \sum_{|k| \leq n-1} K \left( \frac{|k|}{c_{n}} \right) \mathbb{E} (\hat{\gamma}_{k}) e^{i k \lambda} = \frac{1}{2\pi} \sum_{|k| \leq n-1} \left( \frac{n-|k|}{n} \right) K \left( \frac{|k|}{c_{n}} \right) \gamma_{k} e^{i k \lambda}.\]
Since $c_{n} \xrightarrow[n \rightarrow \infty]{} \infty$ and $\lim_{u \rightarrow 0} K(u) = 1$, thanks to dominated convergence theorem and because $\sum_{k} | \gamma(k) | < + \infty$, it is clear that~\eqref{205} is true.

\end{proof}

\begin{proof}[\textbf{Proof of Lemma~\ref{206bis}}]

Let $S_{n}$ and $\tilde{S}_{n}^{m}$ be defined as:
\[S_{n}(\lambda) = \sum_{k=1}^{n} \epsilon_{k} e^{i k \lambda}\]
\[\tilde{S}_{n}^{m}(\lambda) = \sum_{k=1}^{n} \tilde{\epsilon}_{k,m} e^{i k \lambda}.\]

Because $(a+b)^{2} \leq 2a^{2} + 2b^{2}$, we have:
\begin{eqnarray*}
\frac{1}{n} \left \| S_{n}(\lambda) - \tilde{S}_{n}^{m}(\lambda) \right \|_{\mathbb{L}^{2}}^{2} 
&=& \frac{1}{n} \left \| \sum_{k=1}^{n} \epsilon_{k} e^{i k \lambda} - \sum_{k=1}^{n} \tilde{\epsilon}_{k,m} e^{i k \lambda} \right \|_{\mathbb{L}^{2}}^{2} \\
&=& \frac{1}{n} \left \| \sum_{k=1}^{n} \epsilon_{k} e^{i k \lambda} - \left( \sum_{k=1}^{n} \mathbb{E}(\epsilon_{k} | \mathcal{F}_{k+m}) e^{i k \lambda} - \mathbb{E}(\epsilon_{k} | \mathcal{F}_{k-m}) e^{i k \lambda} \right) \right \|_{\mathbb{L}^{2}}^{2} \\
&=& \frac{1}{n} \left \| \sum_{k=1}^{n} (\epsilon_{k} - \mathbb{E}(\epsilon_{k} | \mathcal{F}_{k+m})) e^{i k \lambda} + \sum_{k=1}^{n} \mathbb{E}(\epsilon_{k} | \mathcal{F}_{k-m}) e^{i k \lambda} \right \|_{\mathbb{L}^{2}}^{2} \\
&\leq & \frac{2}{n} \left \| \sum_{k=1}^{n} (\epsilon_{k} - \mathbb{E}(\epsilon_{k} | \mathcal{F}_{k+m})) e^{i k \lambda} \right \|_{\mathbb{L}^{2}}^{2} + \frac{2}{n} \left \| \sum_{k=1}^{n} \mathbb{E}(\epsilon_{k} | \mathcal{F}_{k-m}) e^{i k \lambda} \right \|_{\mathbb{L}^{2}}^{2}.
\end{eqnarray*}
We get for the first term of the right-hand side:
\begin{eqnarray*}
\frac{1}{n} \left \| \sum_{k=1}^{n} (\epsilon_{k} - \mathbb{E}(\epsilon_{k} | \mathcal{F}_{k+m})) e^{i k \lambda} \right \|_{\mathbb{L}^{2}}^{2}
&=& \frac{1}{n} \left \| \sum_{k=1}^{n} \sum_{j=k+m+1}^{\infty} P_{j}(\epsilon_{k}) e^{i k \lambda} \right \|_{\mathbb{L}^{2}}^{2} \\
&=& \frac{1}{n} \left \| \sum_{j=m+2}^{\infty} \sum_{k=1}^{n} P_{j}(\epsilon_{k}) e^{i k \lambda} \textbf{1}_{\{j \geq k+m+1\}} \right \|_{\mathbb{L}^{2}}^{2} \\ 
&=& \frac{1}{n} \sum_{j=m+2}^{\infty} \left \| \sum_{k=1}^{n} P_{j}(\epsilon_{k}) e^{i k \lambda} \textbf{1}_{\{k-j \leq -(m+1)\}} \right \|_{\mathbb{L}^{2}}^{2} \\
&\leq & \frac{1}{n} \sum_{j=m+2}^{\infty} \left( \sum_{k=1}^{n} \left \| P_{j}(\epsilon_{k}) \right \|_{\mathbb{L}^{2}} \textbf{1}_{\{k-j \leq -(m+1)\}} \right)^{2}, \\
\end{eqnarray*}
using the Pythagoras' theorem and the triangle inequality. It follows: 
\begin{eqnarray}
\frac{1}{n} \sum_{j=m+2}^{\infty} \left( \sum_{k=1}^{n} \left \| P_{j}(\epsilon_{k}) \right \|_{\mathbb{L}^{2}} \textbf{1}_{\{k-j \leq -(m+1)\}} \right)^{2}
&\leq & \frac{1}{n} \sum_{j=m+2}^{\infty} \left( \sum_{k=1}^{n} \left \| P_{0}(\epsilon_{k-j}) \right \|_{\mathbb{L}^{2}} \textbf{1}_{\{k-j \leq -(m+1)\}} \right)^{2} \notag \\
&\leq & \frac{1}{n} \sum_{j=m+2}^{\infty} \left( \sum_{r=-\infty}^{-(m+1)} \left \| P_{0}(\epsilon_{r}) \right \|_{\mathbb{L}^{2}} \textbf{1}_{\{1-j \leq r \leq n-j\}} \right)^{2} \notag \\
&\leq & \frac{1}{n} \sum_{j=m+2}^{\infty} \left( \textbf{1}_{\{1-r \leq j \leq n-r\}} \sum_{r=-\infty}^{-(m+1)} \left \| P_{0}(\epsilon_{r}) \right \|_{\mathbb{L}^{2}} \right)^{2} \notag \\
&\leq & \left( \sum_{r=-\infty}^{-(m+1)} \left \| P_{0}(\epsilon_{r}) \right \|_{\mathbb{L}^{2}} \right)^{2}. \label{210}
\end{eqnarray}

With the same arguments, the second term of the right-hand side satisfies the inequality:
\begin{equation}
\frac{1}{n} \left \| \sum_{k=1}^{n} \mathbb{E}(\epsilon_{k} | \mathcal{F}_{k-m}) e^{i k \lambda} \right \|_{\mathbb{L}^{2}}^{2} \leq \left( \sum_{r=m}^{\infty} \left \| P_{0}(\epsilon_{r}) \right \|_{\mathbb{L}^{2}} \right)^{2}.
\label{211}
\end{equation}

Consequently, combining~\eqref{210} and~\eqref{211}, we obtain that:
\[\sup_{\lambda \in [-\pi,\pi]} \frac{1}{n} \left \| S_{n}(\lambda) - \tilde{S}_{n}^{m}(\lambda) \right \|_{\mathbb{L}^{2}}^{2} \leq 2 \left( \sum_{r=-\infty}^{-(m+1)} \left \| P_{0}(\epsilon_{r}) \right \|_{\mathbb{L}^{2}} \right)^{2} + 2 \left( \sum_{r=m}^{\infty} \left \| P_{0}(\epsilon_{r}) \right \|_{\mathbb{L}^{2}} \right)^{2}.\]
Then, since $\sum_{i=-\infty}^{\infty} \left \| P_{0}(\epsilon_{i}) \right \|_{\mathbb{L}^{2}} < +\infty$, we have this first result:
\begin{equation}
\lim_{m \rightarrow \infty} \limsup_{n \rightarrow \infty} \sup_{\lambda \in [-\pi,\pi]} \frac{1}{n} \left \| S_{n}(\lambda) - \tilde{S}_{n}^{m}(\lambda) \right \|_{\mathbb{L}^{2}}^{2} = 0.\\
\label{212}
\end{equation}

Define now the two periodograms corresponding to the quantities $S_{n}$ and $\tilde{S}_{n}^{m}$: 

\[I_{n}(\lambda) = \frac{1}{2 \pi n} \left| S_{n}(\lambda) \right|^{2} = \frac{1}{2\pi} \sum_{k=1-n}^{n-1} \hat{\gamma}_{k} e^{i k \lambda}\]
\[\tilde{I}_{n}^{m}(\lambda) = \frac{1}{2 \pi n} \left| \tilde{S}_{n}^{m}(\lambda) \right|^{2} = \frac{1}{2\pi} \sum_{k=1-n}^{n-1} \hat{\tilde{\gamma}}_{k,m} e^{i k \lambda}.\]
By the Cauchy-Schwarz inequality and the triangle inequality:
\begin{eqnarray*}
\left \| I_{n}(\lambda) - \tilde{I}_{n}^{m}(\lambda) \right \|_{\mathbb{L}^{1}}
&=& \left \| \frac{1}{2 \pi n} \left| S_{n}(\lambda) \right|^{2} - \frac{1}{2 \pi n} \left| \tilde{S}_{n}^{m}(\lambda) \right|^{2} \right \|_{\mathbb{L}^{1}} \\
&=& \frac{1}{2 \pi n} \left \| \left| S_{n}(\lambda) \right|^{2} - \left| \tilde{S}_{n}^{m}(\lambda) \right|^{2} \right \|_{\mathbb{L}^{1}} \\
&=& \frac{1}{2 \pi n} \left \| \left( \left| S_{n}(\lambda) \right| - \left| \tilde{S}_{n}^{m}(\lambda) \right| \right) \left( \left| S_{n}(\lambda) \right| + \left| \tilde{S}_{n}^{m}(\lambda) \right | \right|) \right \|_{\mathbb{L}^{1}} \\
&\leq & \frac{1}{2 \pi n} \left \| \left| S_{n}(\lambda) \right| - \left| \tilde{S}_{n}^{m}(\lambda) \right| \right \|_{\mathbb{L}^{2}} \left \| \left| S_{n}(\lambda) \right| + \left| \tilde{S}_{n}^{m}(\lambda) \right | \right \|_{\mathbb{L}^{2}} \\
&\leq & \frac{1}{2 \pi } \frac{1}{\sqrt{n}} \left \| S_{n}(\lambda) - \tilde{S}_{n}^{m}(\lambda) \right \|_{\mathbb{L}^{2}} \left( \frac{\left \| S_{n}(\lambda) \right \|_{\mathbb{L}^{2}}}{\sqrt{n}} + \frac{ \left \| \tilde{S}_{n}^{m}(\lambda)  \right \|_{\mathbb{L}^{2}}}{\sqrt{n}} \right). \\
\end{eqnarray*}
Thus, thanks to~\eqref{212} and the following inequality for $S_{n}$ and $\tilde{S}_{n}^{m}$:
\begin{equation}
\frac{1}{\sqrt{n}} \left \| S_{n}(\lambda) \right \|_{\mathbb{L}^{2}} \leq \sum_{k \in \mathbb{Z}} \left \| P_{0}(\epsilon_{k}) \right \|_{\mathbb{L}^{2}} < \infty,
\label{215}
\end{equation}
we get:
\begin{equation}
\lim_{m \rightarrow \infty} \limsup_{n \rightarrow \infty} \sup_{\lambda \in [-\pi,\pi]} \left \| I_{n}(\lambda) - \tilde{I}_{n}^{m}(\lambda) \right \|_{\mathbb{L}^{1}} = 0.\\
\label{216}
\end{equation}

Then, let $\hat{K}(.)$ be the Fourier transform of $K$:
\begin{eqnarray*}
f_{n}(\lambda) - \tilde{f}_{n}^{m}(\lambda)
&=& \frac{1}{2\pi} \sum_{|k| \leq n-1} K \left( \frac{|k|}{c_{n}} \right) e^{i k \lambda} \left( \hat{\gamma}_{k} - \hat{\tilde{\gamma}}_{k,m} \right) \\
&=& \frac{1}{2\pi} \sum_{|k| \leq n-1} \frac{1}{2\pi} \left( \int_{\mathbb{R}} \hat{K}(u) e^{i u \frac{k}{c_{n}}} du \right) e^{i k \lambda}  \left( \hat{\gamma}_{k} - \hat{\tilde{\gamma}}_{k,m} \right) \\ 
&=& \frac{1}{2\pi} \int_{\mathbb{R}} \hat{K}(u) \frac{1}{2\pi} \sum_{|k| \leq n-1} \left( \hat{\gamma}_{k} - \hat{\tilde{\gamma}}_{k} \right) e^{i k (\frac{u}{c_{n}} + \lambda)} du \\
&=& \frac{1}{2\pi} \int_{\mathbb{R}} \hat{K}(u) \left( I_{n} \left( \frac{u}{c_{n}} + \lambda \right) - \tilde{I}_{n}^{m} \left( \frac{u}{c_{n}} + \lambda \right) \right) du, \\
\end{eqnarray*}
using the definition of $I_{n}$ and $\tilde{I}_{n}^{m}$. Hence, by the triangle inequality:
\begin{eqnarray*}
\left \| f_{n}(\lambda) - \tilde{f}_{n}^{m}(\lambda) \right \|_{\mathbb{L}^{1}}
&=& \left \| \frac{1}{2\pi} \int_{\mathbb{R}} \hat{K}(u) \left( I_{n} \left( \frac{u}{c_{n}} + \lambda \right) - \tilde{I}_{n}^{m} \left( \frac{u}{c_{n}} + \lambda \right) \right) du \right \|_{\mathbb{L}^{1}} \\
&\leq & \frac{1}{2\pi} \int_{\mathbb{R}} \left| \hat{K}(u) \right| \left \| \left( I_{n} \left( \frac{u}{c_{n}} + \lambda \right) - \tilde{I}_{n}^{m}\left( \frac{u}{c_{n}} + \lambda \right) \right) \right \|_{\mathbb{L}^{1}} du \\
&\leq & \frac{1}{2\pi} \sup_{\theta} \left \| I_{n}(\theta) - \tilde{I}_{n}^{m}(\theta) \right \|_{\mathbb{L}^{1}} \int_{\mathbb{R}} \left| \hat{K}(u) \right| du. \\
\end{eqnarray*}
Using~\eqref{216} and the fact that $\hat{K}$ is integrable, Lemma~\ref{206bis} is proved.

\end{proof}

\begin{proof}[\textbf{Proof of Lemma~\ref{207bis}}]

Without loss of generality, suppose $\theta=0$. We have:
\begin{eqnarray*}
\tilde{f}_{n}^{m}(0) 
&=& \frac{1}{2\pi} \sum_{|k| \leq n-1} K \left( \frac{|k|}{c_{n}} \right) \hat{\tilde{\gamma}}_{k,m} \\
&=& \frac{2}{2\pi} \sum_{k=1}^{n-1} K \left( \frac{k}{c_{n}} \right) \hat{\tilde{\gamma}}_{k,m} + \frac{1}{2 \pi} \hat{\tilde{\gamma}}_{0,m} \\
&=& \frac{2}{2\pi} \sum_{k=1}^{n-1} K \left( \frac{k}{c_{n}} \right) \frac{1}{n} \sum_{j=1}^{n-k} \tilde{\epsilon}_{j,m} \tilde{\epsilon}_{j+k,m}+ \frac{1}{2 \pi n} \sum_{j=1}^{n} \tilde{\epsilon}_{j,m}^{2}.
\end{eqnarray*}

By the triangle inequality again and a change of variables, we have:
\begin{eqnarray*}
&\phantom{=}&
\left \| \tilde{f}_{n}^{m}(0) - \mathbb{E} \left( \tilde{f}_{n}^{m}(0) \right) \right \|_{\mathbb{L}^{1}} \\
&=& \left \| \frac{2}{2\pi} \sum_{k=1}^{n-1} K \left( \frac{k}{c_{n}} \right) \frac{1}{n} \sum_{j=1}^{n-k} \tilde{\epsilon}_{j,m} \tilde{\epsilon}_{j+k,m} -\mathbb{E}(\tilde{\epsilon}_{j,m} \tilde{\epsilon}_{j+k,m}) + \frac{1}{2 \pi n} \sum_{j=1}^{n} \tilde{\epsilon}_{j,m}^{2} - \mathbb{E}(\tilde{\epsilon}_{j,m}^{2}) \right \|_{\mathbb{L}^{1}} \\
&\leq & \frac{2}{2\pi} \left \| \sum_{k=1}^{n-1} K \left( \frac{k}{c_{n}} \right) \frac{1}{n} \sum_{j=1}^{n-k} \left( \tilde{\epsilon}_{j,m} \tilde{\epsilon}_{j+k,m} - \mathbb{E}(\tilde{\epsilon}_{j,m} \tilde{\epsilon}_{j+k,m}) \right) \right \|_{\mathbb{L}^{1}} \\
&& +\: \frac{1}{2\pi} \left \| \frac{1}{n} \sum_{i=1}^{n} \tilde{\epsilon}_{i,m}^{2} - \mathbb{E}(\tilde{\epsilon}_{0,m}^{2}) \right \|_{\mathbb{L}^{1}} \\
&\leq & \frac{2}{2\pi} \left \| \frac{1}{n} \sum_{i=2}^{n} \sum_{j=(i-2c_{n}) \vee 1}^{i-1} K \left( \frac{i-j}{c_{n}} \right) (\tilde{\epsilon}_{i,m} \tilde{\epsilon}_{j,m} - \mathbb{E}(\tilde{\epsilon}_{i,m} \tilde{\epsilon}_{j,m})) \right \|_{\mathbb{L}^{1}} \\
&& +\: \frac{1}{2\pi} \left \| \frac{1}{n} \sum_{i=1}^{n} \tilde{\epsilon}_{i,m}^{2} - \mathbb{E}(\tilde{\epsilon}_{0,m}^{2}) \right \|_{\mathbb{L}^{1}} .
\end{eqnarray*}
By the $\mathbb{L}^{1}$-ergodic theorem, it is known that, at $m$ fixed:
\[\lim_{n \rightarrow \infty} \left \| \frac{1}{n} \sum_{i=1}^{n} \tilde{\epsilon}_{i,m}^{2} - \mathbb{E}(\tilde{\epsilon}_{0,m}^{2}) \right \|_{\mathbb{L}^{1}} = 0.\]
Consequently, it remains to prove:
\[\lim_{m \rightarrow \infty} \limsup_{n \rightarrow \infty} \left \| \frac{1}{n} \sum_{i=2}^{n} \sum_{j=(i-2c_{n}) \vee 1}^{i-1} K \left( \frac{i-j}{c_{n}} \right) (\tilde{\epsilon}_{i,m} \tilde{\epsilon}_{j,m} - \mathbb{E}(\tilde{\epsilon}_{i,m} \tilde{\epsilon}_{j,m})) \right \|_{\mathbb{L}^{1}} = 0.\]

We know that:
\begin{equation}
\frac{1}{n} \sum_{i=2m+1}^{n} \sum_{j=(i-2c_{n}) \vee 1}^{i-2m} K \left( \frac{i-j}{c_{n}} \right) \mathbb{E}(\tilde{\epsilon}_{i,m} \tilde{\epsilon}_{j,m}) = 0.
\label{-100}
\end{equation}
Indeed, \[\mathbb{E}(\tilde{\epsilon}_{i,m} \tilde{\epsilon}_{j,m}) = \mathbb{E} \left( \left( \mathbb{E}(\epsilon_{j} | \mathcal{F}_{j+m}) - \mathbb{E}(\epsilon_{j} | \mathcal{F}_{j-m}) \right) \left( \mathbb{E}(\epsilon_{i} | \mathcal{F}_{i+m}) - \mathbb{E}(\epsilon_{i} | \mathcal{F}_{i-m}) \right) \right).\]
But $\mathbb{E}(\epsilon_{i} | \mathcal{F}_{i+m}) - \mathbb{E}(\epsilon_{i} | \mathcal{F}_{i-m})$ is orthogonal to $\mathbb{L}^{2}(\mathcal{F}_{i-m})$, and $\mathbb{E}(\epsilon_{j} | \mathcal{F}_{j+m}) - \mathbb{E}(\epsilon_{j} | \mathcal{F}_{j-m})$ belongs to $\mathbb{L}^{2}(\mathcal{F}_{i-m})$ if $j+m \leq i-m$. Thus $\mathbb{E}(\tilde{\epsilon}_{i,m} \tilde{\epsilon}_{j,m})$ is equal to zero if $j \leq i-2m$ and~\eqref{-100} is true. 

Thereby we have:
\begin{eqnarray}
&\phantom{=}&
\left \| \frac{1}{n} \sum_{i=2}^{n} \sum_{j=(i-2c_{n}) \vee 1}^{i-1} K \left( \frac{i-j}{c_{n}} \right) (\tilde{\epsilon}_{i,m} \tilde{\epsilon}_{j,m} - \mathbb{E}(\tilde{\epsilon}_{i,m} \tilde{\epsilon}_{j,m})) \right \|_{\mathbb{L}^{1}} \notag \\
&\leq & \left \| \frac{1}{n} \sum_{i=2}^{n} \sum_{j=(i-2m+1) \vee 1}^{i-1} K \left( \frac{i-j}{c_{n}} \right) (\tilde{\epsilon}_{i,m} \tilde{\epsilon}_{j,m} - \mathbb{E}(\tilde{\epsilon}_{i,m} \tilde{\epsilon}_{j,m})) \right \|_{\mathbb{L}^{1}} \notag \\
&& +\: \left \| \frac{1}{n} \sum_{i=2m+1}^{n} \sum_{j=(i-2c_{n}) \vee 1}^{i-2m} K \left( \frac{i-j}{c_{n}} \right) \tilde{\epsilon}_{i,m} \tilde{\epsilon}_{j,m} \right \|_{\mathbb{L}^{1}} \notag \\
&\leq & \left \| \frac{1}{n} \sum_{k=1}^{2m-1} \sum_{i=1}^{n-k} K \left( \frac{k}{c_{n}} \right) (\tilde{\epsilon}_{i,m} \tilde{\epsilon}_{i+k,m} - \mathbb{E}(\tilde{\epsilon}_{i,m} \tilde{\epsilon}_{i+k,m})) \right \|_{\mathbb{L}^{1}} \notag \\
&& +\: \left \| \frac{1}{n} \sum_{i=2m+1}^{n} \sum_{j=(i-2c_{n}) \vee 1}^{i-2m} K \left( \frac{i-j}{c_{n}} \right) \tilde{\epsilon}_{i,m} \tilde{\epsilon}_{j,m} \right \|_{\mathbb{L}^{1}} \label{217}.
\end{eqnarray}

For the first term of the right-hand side of~\eqref{217}, since the kernel $K$ is bounded by $1$, we have by the triangle inequality and the stationarity of the error process:
\[\left \| \frac{1}{n} \sum_{k=1}^{2m-1} \sum_{i=1}^{n-k} K \left( \frac{k}{c_{n}} \right) (\tilde{\epsilon}_{i,m} \tilde{\epsilon}_{i+k,m} - \mathbb{E}(\tilde{\epsilon}_{i,m} \tilde{\epsilon}_{i+k,m})) \right \|_{\mathbb{L}^{1}} \leq \sum_{k=1}^{2m-1} \left \| \frac{1}{n} \sum_{i=1}^{n-k} (\tilde{\epsilon}_{i,m} \tilde{\epsilon}_{i+k,m} - \mathbb{E}(\tilde{\epsilon}_{0,m} \tilde{\epsilon}_{k,m})) \right \|_{\mathbb{L}^{1}}.\]
Using the $\mathbb{L}^{1}$-ergodic theorem, for all $k$ fixed, we deduce that:
\[\sum_{k=1}^{2m-1} \left \| \frac{1}{n} \sum_{i=1}^{n-k} (\tilde{\epsilon}_{i,m} \tilde{\epsilon}_{i+k,m} - \mathbb{E}(\tilde{\epsilon}_{0,m} \tilde{\epsilon}_{k,m})) \right \|_{\mathbb{L}^{1}} \xrightarrow[n \rightarrow \infty]{} 0.\]

It remains to be shown that:
\[\lim_{m \rightarrow \infty} \limsup_{n \rightarrow \infty} \left \| \frac{1}{n} \sum_{i=2m+1}^{n} \sum_{j=(i-2c_{n}) \vee 1}^{i-2m} K \left( \frac{i-j}{c_{n}} \right) \tilde{\epsilon}_{i,m} \tilde{\epsilon}_{j,m} \right \|_{\mathbb{L}^{1}} = 0.\]
We have:
\begin{multline*}
\left \| \frac{1}{n} \sum_{i=2m+1}^{n} \sum_{j=(i-2c_{n}) \vee 1}^{i-2m} K \left( \frac{i-j}{c_{n}} \right) \tilde{\epsilon}_{i,m} \tilde{\epsilon}_{j,m} \right \|_{\mathbb{L}^{1}} \\
= \Bigg \| \frac{1}{n} \sum_{i=2m+1}^{2[n/2m]m} \sum_{j=(i-2c_{n}) \vee 1}^{i-2m} K \left( \frac{i-j}{c_{n}} \right) \tilde{\epsilon}_{i,m} \tilde{\epsilon}_{j,m} \\
+ \frac{1}{n} \sum_{i=2[n/2m]m+1}^{n} \sum_{j=(i-2c_{n}) \vee 1}^{i-2m} K \left( \frac{i-j}{c_{n}} \right) \tilde{\epsilon}_{i,m} \tilde{\epsilon}_{j,m}\Bigg \|_{\mathbb{L}^{1}},
\end{multline*}
then by triangle inequality:
\begin{multline*}
\left \| \frac{1}{n} \sum_{i=2m+1}^{2[n/2m]m} \sum_{j=(i-2c_{n}) \vee 1}^{i-2m} K \left( \frac{i-j}{c_{n}} \right) \tilde{\epsilon}_{i,m} \tilde{\epsilon}_{j,m} + \frac{1}{n} \sum_{i=2[n/2m]m+1}^{n} \sum_{j=(i-2c_{n}) \vee 1}^{i-2m} K \left( \frac{i-j}{c_{n}} \right) \tilde{\epsilon}_{i,m} \tilde{\epsilon}_{j,m}\right \|_{\mathbb{L}^{1}} \\
\leq \left \| \frac{1}{n} \sum_{i=2m+1}^{2[n/2m]m} \sum_{j=(i-2c_{n}) \vee 1}^{i-2m} K \left( \frac{i-j}{c_{n}} \right) \tilde{\epsilon}_{i,m} \tilde{\epsilon}_{j,m} \right \|_{\mathbb{L}^{1}} \\
+ \left \| \frac{1}{n} \sum_{i=2[n/2m]m+1}^{n} \sum_{j=(i-2c_{n}) \vee 1}^{i-2m} K \left( \frac{i-j}{c_{n}} \right) \tilde{\epsilon}_{i,m} \tilde{\epsilon}_{j,m}\right \|_{\mathbb{L}^{1}},
\end{multline*}
and using a change of variable:
\begin{multline}
\left \| \frac{1}{n} \sum_{i=2m+1}^{2[n/2m]m} \sum_{j=(i-2c_{n}) \vee 1}^{i-2m} K \left( \frac{i-j}{c_{n}} \right) \tilde{\epsilon}_{i,m} \tilde{\epsilon}_{j,m} \right \|_{\mathbb{L}^{1}} \\
+ \left \| \frac{1}{n} \sum_{i=2[n/2m]m+1}^{n} \sum_{j=(i-2c_{n}) \vee 1}^{i-2m} K \left( \frac{i-j}{c_{n}} \right) \tilde{\epsilon}_{i,m} \tilde{\epsilon}_{j,m}\right \|_{\mathbb{L}^{1}} \\
\leq \sum_{l=1}^{2m} \left \| \frac{1}{n} \sum_{r=1}^{[n/2m]-1} \tilde{\epsilon}_{2rm+l,m} \sum_{j=(2rm+l-2c_{n}) \vee 1}^{2(r-1)m+l} K \left( \frac{2rm+l-j}{c_{n}} \right) \tilde{\epsilon}_{j,m} \right \|_{\mathbb{L}^{1}} \\
+ \left \| \frac{1}{n} \sum_{i=2[n/2m]m+1}^{n} \tilde{\epsilon}_{i,m} \sum_{j=(i-2c_{n}) \vee 1}^{i-2m} K \left( \frac{i-j}{c_{n}} \right) \tilde{\epsilon}_{j,m} \right \|_{\mathbb{L}^{1}} \label{218}.
\end{multline}

For the second term of the right-hand side of~\eqref{218}, by the Cauchy-Schwarz inequality and by stationarity, we get:
\begin{eqnarray}
\left \| \frac{1}{n} \sum_{i=2[n/2m]m+1}^{n} \tilde{\epsilon}_{i,m} \sum_{j=(i-2c_{n}) \vee 1}^{i-2m} K \left( \frac{i-j}{c_{n}} \right) \tilde{\epsilon}_{j,m} \right \|_{\mathbb{L}^{1}} 
&\leq & \frac{1}{n} \sum_{i=2[n/2m]m+1}^{n} \sum_{j=(i-2c_{n}) \vee 1}^{i-2m} \left \| \tilde{\epsilon}_{i,m} \tilde{\epsilon}_{j,m} \right \|_{\mathbb{L}^{1}} \notag \\
&\leq & \frac{1}{n} \sum_{i=2[n/2m]m+1}^{n} \sum_{j=(i-2c_{n}) \vee 1}^{i-2m} \left \| \tilde{\epsilon}_{0,m} \right \|_{\mathbb{L}^{2}}^{2} \notag \\
&\leq & \frac{4}{n} \sum_{i=2[n/2m]m+1}^{n} \sum_{j=(i-2c_{n}) \vee 1}^{i-2m} \left \| \epsilon_{0} \right \|_{\mathbb{L}^{2}}^{2} \notag \\
&\leq & 16m \frac{c_{n}}{n} \left \| \epsilon_{0} \right \|_{\mathbb{L}^{2}}^{2},
\label{219}
\end{eqnarray}
and~\eqref{219} tends to $0$ as $n$ tends to infinity.

Using ideas developed by Dedecker \cite{dedecker1998central} (see the proof of his Theorem $1$), we study the first term of the right-hand side of~\eqref{218} and we shall prove that it is negligible. Let $Z$ be:
\begin{equation}
Z(r,n,m) = \frac{1}{n} \sum_{r=1}^{[n/2m]-1}  \tilde{\epsilon}_{2rm+l,m} \sum_{j=(2rm+l-2c_{n}) \vee 1}^{2(r-1)m+l} K \left( \frac{2rm+l-j}{c_{n}} \right) \tilde{\epsilon}_{j,m}.
\label{220}
\end{equation}

Let $\varphi$ be the function defined by $\varphi'(0) = \varphi(0) = 0$ and $\varphi''(t) = (1 - \left | t \right |) \textbf{1}_{\{|t| <1\}}$, that is the symmetric function such that, for all $t$ greater or equal to $0$, $\varphi (t) = \frac{1}{6}(1-t)^{3} \textbf{1}_{\{t<1\}} + \frac{1}{2} t - \frac{1}{6}$.

Now, for all $\epsilon > 0$, by the growth of $\varphi$, there exists a constant $C$ such that:
\begin{eqnarray*}
\mathbb{E}(\left | Z(r,n,m) \right |) 
&=& \mathbb{E}(\left | (Z(r,n,m) \right | \textbf{1}_{\{|Z(r,n,m)|>\epsilon \}}) + \mathbb{E}(\left | Z(r,n,m) \right | \textbf{1}_{\{|Z(r,n,m)|<\epsilon \}}) \\
&\leq & C \mathbb{E}(\varphi(Z(r,n,m)) \textbf{1}_{\{|Z(r,n,m)|>\epsilon \}}) + \mathbb{E}(\left | Z(r,n,m) \right | \textbf{1}_{\{|Z(r,n,m)|<\epsilon \}}) \\
&\leq & C \mathbb{E}(\varphi(Z(r,n,m))) + \epsilon,
\end{eqnarray*}
because the function $\varphi$ is positive.

We conclude the proof using Lemma~\ref{222bis}.

\end{proof}

\begin{lem}
In the conditions developed at the end of the previous proof, for all fixed $m$:
\begin{equation}
\lim_{n \rightarrow \infty} \mathbb{E}(\varphi (Z(r,n,m))) = 0.
\label{222}
\end{equation}
\label{222bis}
\end{lem}


\begin{proof}[\textbf{Proof of Lemma~\ref{222bis}}]

To prove that~\eqref{222} is negligible, the two following results are needed:
\begin{lem}
The following inequality holds:
\[| \varphi(x+h) - \varphi(x) - h \varphi'(x) | \leq \psi(h),\]
where:
\[\psi(h) = |h|^2 \textbf{1}_{\{|h| \leq 1\}} + (2|h| -1)\textbf{1}_{\{|h| > 1\}}.\]
\end{lem}

\begin{proof}
The function $\varphi$ is continuous and differentiable in the neighborhood of $0$. Using the Taylor formula, we have the following majorations:
\[| \varphi(x+h) - \varphi(x) - h \varphi'(x) | \leq \frac{|h|^{2}}{2} \sup_{u \in \mathbb{R}} | \varphi''(u) | \leq \frac{|h|^{2}}{2};\]
then, by the triangle inequality:
\[\left | \varphi(x+h) - \varphi(x) - h \varphi'(x) \right | \leq \left | \varphi(x+h) - \varphi(x) \right | + \left | h \right | \left | \varphi'(x) \right | \leq 2 \left | h \right | \sup_{u \in \mathbb{R}} \left | \varphi'(u) \right | \leq \left | h \right |.\]

The proof is complete.

\end{proof}

\begin{lem}
For all real $x$ in $\mathbb{R}$, we have:
\[|x| (1 \wedge |x|) \leq \psi(x) \leq 2|x|(1 \wedge |x|).\]
\label{lem626}
\end{lem}

The proof of Lemma~\ref{lem626}, being elementary, is left to the reader.\\

So we get:
\begin{multline*}
\mathbb{E}(\varphi (Z(r,n,m))) = \sum_{i=1}^{\left[ n/2m \right] -1} \mathbb{E} \Bigg( \Bigg[ \varphi \Bigg( \frac{1}{n} \sum_{q=1}^{i}  \tilde{\epsilon}_{2qm+l,m} \sum_{j=(2qm+l-2c_{n}) \vee 1}^{2(q-1)m+l} K \Bigg( \frac{2qm+l-j}{c_{n}} \Bigg) \tilde{\epsilon}_{j,m} \Bigg) \\
- \varphi \Bigg( \frac{1}{n} \sum_{q=1}^{i-1} \tilde{\epsilon}_{2qm+l,m} \sum_{j=(2qm+l-2c_{n}) \vee 1}^{2(q-1)m+l} K \Bigg( \frac{2qm+l-j}{c_{n}} \Bigg) \tilde{\epsilon}_{j,m} \Bigg) \Bigg] \Bigg) \\
\leq \sum_{i=1}^{\left[ n/2m \right] -1} \Bigg| \mathbb{E} \Bigg( \varphi \Bigg( \frac{1}{n} \sum_{q=1}^{i}  \tilde{\epsilon}_{2qm+l,m} \sum_{j=(2qm+l-2c_{n}) \vee 1}^{2(q-1)m+l} K \Bigg( \frac{2qm+l-j}{c_{n}} \Bigg) \tilde{\epsilon}_{j,m} \Bigg) \\
- \varphi \Bigg( \frac{1}{n} \sum_{q=1}^{i-1} \tilde{\epsilon}_{2qm+l,m} \sum_{j=(2qm+l-2c_{n}) \vee 1}^{2(q-1)m+l} K \Bigg( \frac{2qm+l-j}{c_{n}} \Bigg) \tilde{\epsilon}_{j,m} \Bigg) \Bigg) \Bigg|.
\end{multline*}
Then applying Taylor's expansion, with :
\begin{eqnarray*}
x &=& \frac{1}{n} \sum_{q=1}^{i-1}  \tilde{\epsilon}_{2qm+l,m} \sum_{j=(2qm+l-2c_{n}) \vee 1}^{2(q-1)m+l} K \left( \frac{2qm+l-j}{c_{n}} \right) \tilde{\epsilon}_{j,m} \\
A(i,m) &=& \frac{1}{n} \tilde{\epsilon}_{2im+l,m} \sum_{j=(2im+l-2c_{n}) \vee 1}^{2(i-1)m+l} K \left( \frac{2im+l-j}{c_{n}} \right) \tilde{\epsilon}_{j,m} \\
x+A(i,m) &=& \frac{1}{n} \sum_{q=1}^{i}  \tilde{\epsilon}_{2qm+l,m} \sum_{j=(2qm+l-2c_{n}) \vee 1}^{2(q-1)m+l} K \left( \frac{2qm+l-j}{c_{n}} \right) \tilde{\epsilon}_{j,m},
\end{eqnarray*}
we have:
\begin{eqnarray*}
\mathbb{E}(\varphi (Z(r,n,m))) 
&\leq & \sum_{i=1}^{[n/2m]-1} \Bigg| \mathbb{E} \Bigg( \varphi' \left( \frac{1}{n} \sum_{q=1}^{i-1}  \tilde{\epsilon}_{2qm+l,m} \sum_{j=(2qm+l-2c_{n}) \vee 1}^{2(q-1)m+l} K \left( \frac{2qm+l-j}{c_{n}} \right) \tilde{\epsilon}_{j,m} \right) \\
&& \qquad \times\: \frac{1}{n} \tilde{\epsilon}_{2im+l,m} \sum_{j=(2im+l-2c_{n}) \vee 1}^{2(i-1)m+l} K \left( \frac{2im+l-j}{c_{n}} \right) \tilde{\epsilon}_{j,m} + \psi(A(i,m)) \Bigg) \Bigg|.
\end{eqnarray*}
Then by triangle inequality, we obtain:
\begin{eqnarray*}
\mathbb{E}(\varphi (Z(r,n,m))) 
&\leq & \sum_{i=1}^{[n/2m]-1} \Bigg| \mathbb{E} \Bigg( \varphi' \left( \frac{1}{n} \sum_{q=1}^{i-1}  \tilde{\epsilon}_{2qm+l,m} \sum_{j=(2qm+l-2c_{n}) \vee 1}^{2(q-1)m+l} K \left( \frac{2qm+l-j}{c_{n}} \right) \tilde{\epsilon}_{j,m} \right) \\ 
&& \quad \times\: \frac{1}{n} \tilde{\epsilon}_{2im+l,m} \sum_{j=(2im+l-2c_{n}) \vee 1}^{2(i-1)m+l} K \left( \frac{2im+l-j}{c_{n}} \right) \tilde{\epsilon}_{j,m} \Bigg) \Bigg| \\
&& \qquad + \sum_{i=1}^{[n/2m]-1} \Bigg| \mathbb{E} \left( \left| A(i,m)  \right|^{2} \textbf{1}_{\{|A(i,m)| \leq 1\}} + \left( 2 \left| A(i,m) \right| -1 \right) \textbf{1}_{\{|A(i,m)|>1\}} \right) \Bigg| \\
&\leq & \sum_{i=1}^{[n/2m]-1} \Bigg| \mathbb{E} \Bigg( \varphi' \left( \frac{1}{n} \sum_{q=1}^{i-1}  \tilde{\epsilon}_{2qm+l,m} \sum_{j=(2qm+l-2c_{n}) \vee 1}^{2(q-1)m+l} K \left( \frac{2qm+l-j}{c_{n}} \right) \tilde{\epsilon}_{j,m} \right) \\
&& \quad \times\: \frac{1}{n} \tilde{\epsilon}_{2im+l,m} \sum_{j=(2im+l-2c_{n}) \vee 1}^{2(i-1)m+l} K \left( \frac{2im+l-j}{c_{n}} \right) \tilde{\epsilon}_{j,m} \Bigg) \Bigg| \\
&& \qquad + \sum_{i=1}^{[n/2m]-1} \mathbb{E} \left( \left| A(i,m)  \right|^{2} \textbf{1}_{\{|A(i,m)| \leq 1\}} + \left( 2 \left| A(i,m) \right| -1 \right) \textbf{1}_{\{|A(i,m)|>1\}} \right). \\
\end{eqnarray*}

By definition, $(\tilde{\epsilon}_{i,m})_{i \in \mathbb{Z}}$ satisfies:
\[\mathbb{E}(\tilde{\epsilon}_{2im+l,m} | \mathcal{F}_{2im+l-m}) = 0.\]
Hence:
\begin{multline*}
\mathbb{E}(\varphi (Z(r,n,m))) \leq \sum_{i=1}^{[n/2m]-1} \mathbb{E} \left( | A(i,m) |^{2} \textbf{1}_{\{|A(i,m)| \leq 1\}} + (2|A(i,m)| -1) \textbf{1}_{\{|A(i,m)|>1\}} \right) \\
= \sum_{i=1}^{[n/2m]-1} \mathbb{E}( \psi ( | A(i,m) | ).
\end{multline*}

For this term, put:
\[B(i,j,m,l) = \frac{[(2(i-1)m+l) - ((2im+l-2c_{n}) \vee 1)+1]}{n} K \left( \frac{2im+l-j}{c_{n}} \right) \tilde{\epsilon}_{2im+l,m} \tilde{\epsilon}_{j,m}.\]
Using the convexity of $\psi$ and Lemma 3 of Dedecker \cite{dedecker1998central}, we have that:
\begin{multline*}
\mathbb{E}(\psi(A(i,m))) \\
\leq \frac{1}{[(2(i-1)m+l) - ((2im+l-2c_{n}) \vee 1)+1]} \sum_{j=(2im+l-2c_{n}) \vee 1}^{2(i-1)m+l} \mathbb{E} \left( \psi \left( B(i,j,m,l) \right) \right).
\end{multline*}
Then:
\begin{multline*}
\frac{1}{[(2(i-1)m+l) - ((2im+l-2c_{n}) \vee 1)+1]} \sum_{j=(2im+l-2c_{n}) \vee 1}^{2(i-1)m+l} \mathbb{E} \left( \psi \left( B(i,j,m,l) \right) \right) \\
\leq \frac{2}{[(2(i-1)m+l) - ((2im+l-2c_{n}) \vee 1)+1]} \\
\sum_{j=(2im+l-2c_{n}) \vee 1}^{2(i-1)m+l} \mathbb{E} \left( \frac{2c_{n}}{n} | \tilde{\epsilon}_{0,m} |^{2} \left( 1 \wedge \frac{2c_{n}}{n} | \tilde{\epsilon}_{0,m} |^{2} \right) \right),
\end{multline*}
and:
\begin{multline*}
\frac{2}{[(2(i-1)m+l) - ((2im+l-2c_{n}) \vee 1)+1]} \sum_{j=(2im+l-2c_{n}) \vee 1}^{2(i-1)m+l} \mathbb{E} \left( \frac{2c_{n}}{n} | \tilde{\epsilon}_{0,m} |^{2} \left( 1 \wedge \frac{2c_{n}}{n} | \tilde{\epsilon}_{0,m} |^{2} \right) \right) \\
\leq \frac{8c_{n}}{n} \mathbb{E} \left( | \tilde{\epsilon}_{0,m} |^{2} \left( 1 \wedge \frac{c_{n}}{n} | \tilde{\epsilon}_{0,m} |^{2} \right) \right).
\end{multline*}

Thus we can conclude if, for $m$ fixed:
\begin{equation}
\lim_{n \rightarrow \infty} c_{n} \mathbb{E} \left( | \tilde{\epsilon}_{0,m} |^{2} \left( 1 \wedge \frac{c_{n}}{n} | \tilde{\epsilon}_{0,m} |^{2} \right) \right) = 0.
\label{223}
\end{equation}

To prove~\eqref{223}, notice that:
\begin{eqnarray}
c_{n} \mathbb{E} \left( | \tilde{\epsilon}_{0,m} |^{2} \left( 1 \wedge \frac{c_{n}}{n} | \tilde{\epsilon}_{0,m} |^{2} \right) \right)
&\leq & 4c_{n} \mathbb{E} \left( | \mathbb{E}(\epsilon_{0}|\mathcal{F}_{m}) |^{2} \left( 1 \wedge \frac{c_{n}}{n} | \mathbb{E}(\epsilon_{0}|\mathcal{F}_{m}) |^{2} \right) \right) \notag \\
&& +\: 4c_{n} \mathbb{E} \left( | \mathbb{E}(\epsilon_{0}|\mathcal{F}_{m}) |^{2} \left( 1 \wedge \frac{c_{n}}{n} | \mathbb{E}(\epsilon_{0}|\mathcal{F}_{-m}) |^{2} \right) \right) \notag \\
&& \quad + 4c_{n} \mathbb{E} \left( | \mathbb{E}(\epsilon_{0}|\mathcal{F}_{-m}) |^{2} \left( 1 \wedge \frac{c_{n}}{n} | \mathbb{E}(\epsilon_{0}|\mathcal{F}_{m}) |^{2} \right) \right)\notag \\
&& \qquad + 4c_{n} \mathbb{E} \left( | \mathbb{E}(\epsilon_{0}|\mathcal{F}_{-m}) |^{2} \left( 1 \wedge \frac{c_{n}}{n} | \mathbb{E}(\epsilon_{0}|\mathcal{F}_{-m}) |^{2} \right) \right).
\label{224}
\end{eqnarray}

For the first term and for the last term, we use the convexity of $\psi$:
\begin{eqnarray}
c_{n} \mathbb{E} \left( | \mathbb{E}(\epsilon_{0}|\mathcal{F}_{m}) |^{2} \left( 1 \wedge \frac{c_{n}}{n} | \mathbb{E}(\epsilon_{0}|\mathcal{F}_{m}) |^{2} \right) \right) 
&\leq & n \mathbb{E} \left( \psi \left( \mathbb{E} \left( \frac{c_{n}}{n} |\epsilon_{0}|^{2}|\mathcal{F}_{m} \right) \right) \right) \notag \\
&\leq & n \mathbb{E} \left( \mathbb{E} \left( \psi \left(\frac{c_{n}}{n} |\epsilon_{0}|^{2} \right)|\mathcal{F}_{m} \right) \right) \notag \\
&\leq & n \mathbb{E} \left( \psi \left( \frac{c_{n}}{n} |\epsilon_{0}|^{2} \right) \right) \notag \\
&\leq & 2c_{n} \mathbb{E} \left( |\epsilon_{0}|^{2} \left( 1 \wedge \frac{c_{n}}{n} |\epsilon_{0}|^{2} \right) \right). \label{225}
\end{eqnarray}
With the same idea, for the last term, we show that:
\begin{equation}
c_{n} \mathbb{E} \left( |\mathbb{E}(\epsilon_{0}|\mathcal{F}_{-m})|^{2} \left( 1 \wedge \frac{c_{n}}{n} |\mathbb{E}(\epsilon_{0}|\mathcal{F}_{-m})|^{2} \right) \right) \leq 2c_{n} \mathbb{E} \left( |\epsilon_{0}|^{2} \left( 1 \wedge \frac{c_{n}}{n} |\epsilon_{0}|^{2} \right) \right).
\label{226}
\end{equation}

For the second term, with the convexity of $\psi$, we have:
\begin{eqnarray}
n \mathbb{E} \left( \frac{c_{n}}{n} |\mathbb{E}(\epsilon_{0}|\mathcal{F}_{m})|^{2} \left( 1 \wedge \frac{c_{n}}{n} |\mathbb{E}(\epsilon_{0}|\mathcal{F}_{-m})|^{2} \right) \right) 
&\leq & n \mathbb{E} \left( \frac{c_{n}}{n} \mathbb{E}(|\epsilon_{0}|^{2}|\mathcal{F}_{m}) \left( 1 \wedge \frac{c_{n}}{n} \mathbb{E}(|\epsilon_{0}|^{2}|\mathcal{F}_{-m}) \right) \right) \notag \\
&\leq & n \mathbb{E} \left( \psi \left( \mathbb{E} \left( \frac{c_{n}}{n}|\epsilon_{0}|^{2} | \mathcal{F}_{-m} \right) \right) \right) \notag \\
&\leq & 2c_{n} \mathbb{E} \left( |\epsilon_{0}|^{2} \left( 1 \wedge \frac{c_{n}}{n} |\epsilon_{0}|^{2} \right) \right). \label{227}
\end{eqnarray}

Since $g : x \rightarrow 1 \wedge x$ is a concave function on $\mathbb{R}_{+}^{\ast}$ and $\psi$ is a convex function, for the third term, we obtain that:

\begin{eqnarray}
&\phantom{=}&
c_{n} \mathbb{E} \left( |\mathbb{E}(\epsilon_{0}|\mathbb{F}_{-m})|^{2} \left( 1 \wedge \frac{c_{n}}{n} |\mathbb{E}(\epsilon_{0}|\mathcal{F}_{m})|^{2} \right) \right) \notag \\
&\leq & n \mathbb{E} \left(\frac{c_{n}}{n} \mathbb{E}(|\epsilon_{0}|^{2}|\mathbb{F}_{-m}) \mathbb{E} \left( g \left( \frac{c_{n}}{n} \mathbb{E}(|\epsilon_{0}|^{2} | \mathcal{F}_{m}) \right) | \mathcal{F}_{-m} \right) \right) \notag \\
&\leq & n \mathbb{E} \left( \frac{c_{n}}{n} \mathbb{E}(|\epsilon_{0}|^{2}|\mathbb{F}_{-m}) g \left( \mathbb{E} \left( \frac{c_{n}}{n} \mathbb{E}(|\epsilon_{0}|^{2} | \mathcal{F}_{m}) | \mathcal{F}_{-m} \right) \right) \right) \notag \\
&\leq & n \mathbb{E} \left( \frac{c_{n}}{n} \mathbb{E}(|\epsilon_{0}|^{2}|\mathbb{F}_{-m}) \left( 1 \wedge \frac{c_{n}}{n} \mathbb{E}(|\epsilon_{0}|^{2} | \mathcal{F}_{-m}) \right) \right) \notag \\
&\leq & n \mathbb{E} \left( \psi \left( \frac{c_{n}}{n} \mathbb{E}(|\epsilon_{0}|^{2}|\mathbb{F}_{-m}) \right) \right) \notag \\
&\leq & 2c_{n} \mathbb{E} \left( |\epsilon_{0}|^{2} \left( 1 \wedge \frac{c_{n}}{n} |\epsilon_{0}|^{2} \right) \right). \label{228}
\end{eqnarray}

Using~\eqref{224} to~\eqref{228}, we deduce that \eqref{223} is verified as soon as~\eqref{48} is true.

\end{proof}

\subsubsection{Proposition~\ref{201bis}}

\begin{proof}[Proof]

Recall that:
\[f_{n}(\lambda) = \frac{1}{2\pi} \sum_{|k| \leq n-1} K \left( \frac{|k|}{c_{n}} \right) \hat{\gamma}_{k} e^{i k \lambda},\]
where:
\[\hat{\gamma}_{k} = \frac{1}{n} \sum_{j=1}^{n-|k|} \epsilon_{j} \epsilon_{j+|k|} = \frac{1}{n} \sum_{j=1}^{n-|k|} \left( Y_{j} - \sum_{l=1}^{p} x_{j,l} \beta_{l} \right) \left( Y_{j+|k|} - \sum_{l=1}^{p} x_{j+|k|,l} \beta_{l} \right), \ \quad \ 0 \leq |k| \leq (n-1),\]
and:
\[f_{n}^{\ast}(\lambda) = \frac{1}{2\pi} \sum_{|k| \leq n-1} K \left( \frac{|k|}{c_{n}} \right) \hat{\gamma}_{k}^{\ast} e^{i k \lambda},\]
where:
\[\hat{\gamma}_{k}^{\ast} = \frac{1}{n} \sum_{j=1}^{n-|k|} \hat{\epsilon}_{j} \hat{\epsilon}_{j+|k|} = \frac{1}{n} \sum_{j=1}^{n-|k|} \left( Y_{j} - \sum_{l=1}^{p} x_{j,l} \hat{\beta}_{l} \right) \left( Y_{j+|k|} - \sum_{l=1}^{p} x_{j+|k|,l} \hat{\beta}_{l} \right), \ \quad \ 0 \leq |k| \leq (n-1).\]
Thus we have:
\begin{eqnarray}
\left \| f_{n}^{\ast}(\lambda) - f_{n}(\lambda) \right \|_{\mathbb{L}^{1}} 
&=& \left \| \frac{1}{2\pi} \sum_{|k| \leq n-1} K \left( \frac{|k|}{c_{n}} \right) \hat{\gamma}_{k}^{\ast} e^{i k \lambda} - \frac{1}{2\pi} \sum_{|k| \leq n-1} K \left( \frac{|k|}{c_{n}} \right) \hat{\gamma}_{k} e^{i k \lambda} \right \|_{\mathbb{L}^{1}} \notag \\
&=& \left \| \frac{1}{2\pi} \sum_{|k| \leq 2c_{n}} K \left( \frac{|k|}{c_{n}} \right) \left[ \hat{\gamma}_{k}^{\ast} - \hat{\gamma}_{k} \right] e^{i k \lambda}\right \|_{\mathbb{L}^{1}} \notag \\
&\leq & \frac{1}{2\pi} \sum_{|k| \leq 2c_{n}} \left \| \hat{\gamma}_{k}^{\ast} - \hat{\gamma}_{k} \right \|_{\mathbb{L}^{1}}. \label{250}
\end{eqnarray}
Since $\frac{c_{n}}{n}$ tends to $0$ when $n$ tends to infinity, it remains to prove that:
\begin{equation}
\sup_{|k| \leq 2c_{n}} \left \| \hat{\gamma}_{k}^{\ast} - \hat{\gamma}_{k} \right \|_{\mathbb{L}^{1}} = \mathcal{O} \left( \frac{1}{n} \right).
\label{251}
\end{equation}

\begin{lem}
The following inequality is verified:
\begin{eqnarray}
\left \| \hat{\gamma}_{k}^{\ast} - \hat{\gamma}_{k} \right \|_{\mathbb{L}^{1}} 
&=& \Bigg \| \frac{1}{n} \sum_{j=1}^{n-|k|} \left( Y_{j} - \sum_{l=1}^{p} x_{j,l} \hat{\beta}_{l} \right) \left( Y_{j+|k|} - \sum_{l=1}^{p} x_{j+|k|,l} \hat{\beta}_{l} \right) \notag \\
&& -\: \frac{1}{n} \sum_{j=1}^{n-|k|} \left( Y_{j} - \sum_{l=1}^{p} x_{j,l} \beta_{l} \right) \left( Y_{j+|k|} - \sum_{l=1}^{p} x_{j+|k|,l} \beta_{l} \right) \Bigg \|_{\mathbb{L}^{1}} \notag \\
&\leq & \frac{1}{2n} \sum_{l=1}^{p} \sum_{l'=1}^{p} \left \| \left( \beta_{l} - \hat{\beta_{l}} \right)^{2} \sum_{j=1}^{n-|k|} x_{j,l}^{2}  \right \|_{\mathbb{L}^{1}} + \frac{1}{2n} \sum_{l=1}^{p} \sum_{l'=1}^{p} \left \|  \left(  \beta_{l'} - \hat{\beta}_{l'} \right)^{2} \sum_{j=1}^{n-|k|} x_{j+|k|,l'}^{2}  \right \|_{\mathbb{L}^{1}} \notag \\ 
&& +\: \frac{1}{n} \sum_{l=1}^{p} \left \| \sum_{j=1}^{n-|k|} \epsilon_{j} x_{j+|k|,l} \left(  \beta_{l} - \hat{\beta}_{l} \right) \right \|_{\mathbb{L}^{1}} + \frac{1}{n} \sum_{l=1}^{p} \left \| \sum_{j=1}^{n-|k|} \epsilon_{j+|k|} x_{j,l} \left( \beta_{l} - \hat{\beta_{l}} \right) \right \|_{\mathbb{L}^{1}}. \label{252}
\end{eqnarray}
\label{6.2.7}
\end{lem}
The proof of this lemma will be given in Section $5.3$.\\

It remains to calculate these four terms. For the first term of the right-hand side, for all $l$, $l'$ fixed and for all $k$, we have:
\[\left \| \left( \beta_{l} - \hat{\beta_{l}} \right)^{2} \sum_{j=1}^{n-|k|} x_{j,l}^{2}  \right \|_{\mathbb{L}^{1}} \leq \left \| \left( \beta_{l} - \hat{\beta_{l}} \right)^{2} \sum_{j=1}^{n} x_{j,l}^{2}  \right \|_{\mathbb{L}^{1}},\]
and:
\[\left \| \left( \beta_{l} - \hat{\beta_{l}} \right)^{2} \sum_{j=1}^{n} x_{j,l}^{2}  \right \|_{\mathbb{L}^{1}} = \left \| d_{l}(n)^{2} \left( \beta_{l} - \hat{\beta_{l}} \right)^{2}  \right \|_{\mathbb{L}^{1}} = d_{l}(n)^{2} \mathbb{E} \left( \left( \beta_{l} - \hat{\beta_{l}} \right)^{2}  \right).\]
Hannan has proven in his paper \cite{hannan73clt} a Central Limit Theorem~\eqref{15} with the convergence of the second order moments~\eqref{15bis}. Consequently, we have:
\[\left \| d_{l}(n)^{2} \left( \beta_{l} - \hat{\beta_{l}} \right)^{2}  \right \|_{\mathbb{L}^{1}} = \mathcal{O}(1),\]
hence:
\[\sup_{|k| \leq 2c_{n}} \left( \left \| \left( \beta_{l} - \hat{\beta_{l}} \right)^{2} \sum_{j=1}^{n-|k|} x_{j,l}^{2}  \right \|_{\mathbb{L}^{1}} \right) \leq d_{l}(n)^{2} \mathbb{E} \left( \left( \hat{\beta}_{l} - \beta_{l} \right)^{2}  \right) = \mathcal{O}(1).\]
So we can conclude:
\[\sup_{|k| \leq 2c_{n}} \left( \frac{1}{2n} \sum_{l=1}^{p} \sum_{l'=1}^{p} \left \| \left( \beta_{l} - \hat{\beta_{l}} \right)^{2} \sum_{j=1}^{n-|k|} x_{j,l}^{2}  \right \|_{\mathbb{L}^{1}} \right) = \mathcal{O} \left( \frac{1}{n} \right).\]

For the second term, the same arguments are used, because $\sum_{j=1}^{n-|k|} x_{j+|k|,l}^{2} \leq \sum_{j=1}^{n} x_{j,l}^{2}$. Hence:
\[\sup_{|k| \leq 2c_{n}} \left( \frac{1}{2n} \sum_{l=1}^{p} \sum_{l'=1}^{p} \left \|  \left(  \beta_{l'} - \hat{\beta}_{l'} \right)^{2} \sum_{j=1}^{n-|k|} x_{j+|k|,l'}^{2}  \right \|_{\mathbb{L}^{1}} \right) = \mathcal{O} \left( \frac{1}{n} \right).\]

For the third term, for all $l$ fixed, by the Cauchy-Schwarz inequality, we get:
\[\left \| \sum_{j=1}^{n-|k|} \epsilon_{j} x_{j+|k|,l} \left(  \beta_{l} - \hat{\beta}_{l} \right) \right \|_{\mathbb{L}^{1}} \leq \left \| \sum_{j=1}^{n-|k|} \epsilon_{j} x_{j+|k|,l} \right \|_{\mathbb{L}^{2}} \left \| \beta_{l} - \hat{\beta}_{l} \right \|_{\mathbb{L}^{2}}.\]
Then, we have:
\begin{eqnarray*}
\left \| \sum_{j=1}^{n-|k|} \epsilon_{j} x_{j+|k|,l} \right \|_{\mathbb{L}^{2}}^{2} 
&=& \sum_{j=1}^{n-|k|} \sum_{i=1}^{n-|k|} \gamma_{i-j} x_{i+|k|,l} x_{j+|k|,l} \\
&=& \sum_{i=1}^{n-|k|} \sum_{j=i}^{n-|k|} \gamma_{j-i} x_{i+|k|,l} x_{j+|k|,l} + \sum_{i=1}^{n-|k|} \sum_{j=1}^{i-1} \gamma_{j-i} x_{i+|k|,l} x_{j+|k|,l}.
\end{eqnarray*}
For the first term of the right-hand side, it follows with the change of variables $r=j-i$: 
\begin{eqnarray*}
\sum_{i=1}^{n-|k|} \sum_{j=i}^{n-|k|} \gamma_{j-i} x_{i+|k|,l} x_{j+|k|,l} 
&=& \sum_{i=1}^{n-|k|} \sum_{r=0}^{n-|k|-i} \gamma_{r} x_{i+|k|,l} x_{i+|k|+r,l} \\
&\leq & \sum_{i=1}^{n-|k|} \sum_{r=0}^{n-|k|-i} | \gamma_{r} | | x_{i+|k|,l} |  |x_{i+|k|+r,l} | \\
&\leq & \sum_{i=1}^{n-|k|} \sum_{r=0}^{n-|k|-i} | \gamma_{r} | ( x_{i+|k|,l}^{2} + x_{i+|k|+r,l}^{2}) \\
&\leq & \sum_{i=1}^{n-|k|} \sum_{r=0}^{n-|k|-i} | \gamma_{r} | x_{i+|k|,l}^{2} +  \sum_{i=1}^{n-|k|} \sum_{r=0}^{n-|k|-i} | \gamma_{r} | x_{i+|k|+r,l}^{2}. \\
\end{eqnarray*}
Since $r \leq n-|k|-i$, we have $i \leq n-|k|-r$, and it follows that:
\begin{multline*}
\sum_{i=1}^{n-|k|} \sum_{r=0}^{n-|k|-i} | \gamma_{r} | x_{i+|k|,l}^{2} +  \sum_{i=1}^{n-|k|} \sum_{r=0}^{n-|k|-i} | \gamma_{r} | x_{i+|k|+r,l}^{2} \\
\leq \sum_{i=1}^{n-|k|} x_{i+|k|,l}^{2} \sum_{r=0}^{n-|k|-i} | \gamma_{r} |  +  \sum_{r=0}^{n-|k|} | \gamma_{r} | \sum_{i=1}^{n-|k|-r} x_{i+|k|+r,l}^{2} \\
\leq \sum_{i=1}^{n-|k|} x_{i+|k|,l}^{2} \sum_{r}^{} | \gamma_{r} |  +  \sum_{r}^{} | \gamma_{r} | \sum_{i=1}^{n-|k|-r} x_{i+|k|+r,l}^{2}.
\end{multline*}
Since $\sum_{k} | \gamma(k) | < \infty$:
\begin{eqnarray*}
\sum_{i=1}^{n-|k|} x_{i+|k|,l}^{2} \sum_{r}^{} | \gamma_{r} |  +  \sum_{r}^{} | \gamma_{r} | \sum_{i=1}^{n-|k|-r} x_{i+|k|+r,l}^{2}
&\leq & M \left( \sum_{i=1}^{n-|k|} x_{i+|k|,l}^{2} + \sum_{i=1}^{n-|k|-r} x_{i+|k|+r,l}^{2} \right) \\
&\leq & M \left( \sum_{i=1}^{n} x_{i,l}^{2} + \sum_{i=1}^{n} x_{i,l}^{2} \right) \\
&\leq & M' \sum_{i=1}^{n} x_{i,l}^{2}.
\end{eqnarray*}
With the same idea, for the second term of the right-hand side, we have:
\[\sum_{i=1}^{n-|k|} \sum_{j=1}^{i-1} \gamma_{j-i} x_{i+|k|,l} x_{j+|k|,l} \leq M' \sum_{j=1}^{n} x_{j,l}^{2},\]
thus:
\[\sup_{|k| \leq 2c_{n}} \left \| \sum_{j=1}^{n-|k|} \epsilon_{j} x_{j+|k|,l} \right \|_{\mathbb{L}^{2}}^{2} \leq 2 M' \sum_{j=1}^{n} x_{j,l}^{2} = M'' d_{l}(n)^{2}.\]
In conclusion :
\begin{eqnarray*}
\left \| \sum_{j=1}^{n-|k|} \epsilon_{j} x_{j+|k|,l} \left(  \beta_{l} - \hat{\beta}_{l} \right) \right \|_{\mathbb{L}^{1}} 
&\leq & \left \| \sum_{j=1}^{n-|k|} \epsilon_{j} x_{j+|k|,l} \right \|_{\mathbb{L}^{2}} \left \| \beta_{l} - \hat{\beta}_{l} \right \|_{\mathbb{L}^{2}} \\
&\leq & C d_{l}(n) \sqrt{ \mathbb{E} \left( (\beta_{l} - \hat{\beta}_{l})^{2} \right)} \\
&\leq & C \sqrt{ d_{l}(n)^{2} \mathbb{E} \left( (\beta_{l} - \hat{\beta}_{l})^{2} \right)} = \mathcal{O}(1),
\end{eqnarray*}
hence:
\[\sup_{|k| \leq 2c_{n}} \left \| \sum_{j=1}^{n-|k|} \epsilon_{j} x_{j+|k|,l} \left(  \beta_{l} - \hat{\beta}_{l} \right) \right \|_{\mathbb{L}^{1}} = \mathcal{O}(1),\]
therefore:
\[\sup_{|k| \leq 2c_{n}} \left( \frac{1}{n} \sum_{l=1}^{p} \left \| \sum_{j=1}^{n-|k|} \epsilon_{j} x_{j+|k|,l} \left(  \beta_{l} - \hat{\beta}_{l} \right) \right \|_{\mathbb{L}^{1}} \right) = \mathcal{O} \left( \frac{1}{n} \right).\] 

The same idea is used for the fourth term of the right-hand side of~\eqref{252}. Thus~\eqref{251} is verified and consequently~\eqref{201} is true.

\end{proof}

\subsection{Proof of Lemma~\ref{6.2.7}}
 
We start by developing the term $Y_{j}$:
 
\begin{multline*}
\Big \| \hat{\gamma}_{k}^{\ast} - \hat{\gamma}_{k} \Big \|_{\mathbb{L}^{1}}
= \Bigg \| \frac{1}{n} \sum_{j=1}^{n-|k|} \Bigg( Y_{j} - \sum_{l=1}^{p} x_{j,l} \hat{\beta}_{l} \Bigg) \Bigg( Y_{j+|k|} - \sum_{l=1}^{p} x_{j+|k|,l} \hat{\beta}_{l} \Bigg) \\
- \frac{1}{n} \sum_{j=1}^{n-|k|} \Bigg( Y_{j} - \sum_{l=1}^{p} x_{j,l} \beta_{l} \Bigg) \Bigg( Y_{j+|k|} - \sum_{l=1}^{p} x_{j+|k|,l} \beta_{l} \Bigg) \Bigg \|_{\mathbb{L}^{1}} \\
= \Bigg \| \frac{1}{n} \sum_{j=1}^{n-|k|} \Bigg( \sum_{l=1}^{p} x_{j,l} \Big( \beta_{l} - \hat{\beta_{l}} \Big) + \epsilon_{j} \Bigg) \Bigg( \sum_{l=1}^{p} x_{j+|k|,l} \Big( \beta_{l} - \hat{\beta}_{l} \Big) + \epsilon_{j+|k|} \Bigg) \\
- \frac{1}{n} \sum_{j=1}^{n-|k|} \Bigg( Y_{j} - \sum_{l=1}^{p} x_{j,l} \beta_{l} \Bigg) \Bigg( Y_{j+|k|} - \sum_{l=1}^{p} x_{j+|k|,l} \beta_{l} \Bigg) \Bigg \|_{\mathbb{L}^{1}}.
\end{multline*} 
Because $\epsilon_{j}$ is equal to $Y_{j} - \sum_{l=1}^{p} x_{j,l} \beta_{l}$, we have:
\begin{multline*}
\Bigg \| \frac{1}{n} \sum_{j=1}^{n-|k|} \Bigg( \sum_{l=1}^{p} x_{j,l} \Big( \beta_{l} - \hat{\beta_{l}} \Big) + \epsilon_{j} \Bigg) \Bigg( \sum_{l=1}^{p} x_{j+|k|,l} \Big( \beta_{l} - \hat{\beta}_{l} \Big) + \epsilon_{j+|k|} \Bigg) \\
- \frac{1}{n} \sum_{j=1}^{n-|k|} \Bigg( Y_{j} - \sum_{l=1}^{p} x_{j,l} \beta_{l} \Bigg) \Bigg( Y_{j+|k|} - \sum_{l=1}^{p} x_{j+|k|,l} \beta_{l} \Bigg) \Bigg \|_{\mathbb{L}^{1}} \\
= \Bigg \| \frac{1}{n} \sum_{j=1}^{n-|k|} \Bigg( \sum_{l=1}^{p} x_{j,l} \Big( \beta_{l} - \hat{\beta_{l}} \Big) \sum_{l=1}^{p} x_{j+|k|,l} \Big(  \beta_{l} - \hat{\beta}_{l} \Big) \\
+ \epsilon_{j} \sum_{l=1}^{p} x_{j+|k|,l} \Big( \beta_{l} - \hat{\beta}_{l} \Big) + \sum_{l=1}^{p} x_{j,l} \Big( \beta_{l} - \hat{\beta_{l}} \Big) \epsilon_{j+|k|} \Bigg) \Bigg \|_{\mathbb{L}^{1}}.
\end{multline*}
Using the triangle inequality, we obtain:
\begin{multline*}
\Bigg \| \frac{1}{n} \sum_{j=1}^{n-|k|} \Bigg( \sum_{l=1}^{p} x_{j,l} \Big( \beta_{l} - \hat{\beta_{l}} \Big) \sum_{l=1}^{p} x_{j+|k|,l} \Big(  \beta_{l} - \hat{\beta}_{l} \Big) \\
+ \epsilon_{j} \sum_{l=1}^{p} x_{j+|k|,l} \Big( \beta_{l} - \hat{\beta}_{l} \Big) + \sum_{l=1}^{p} x_{j,l} \Big( \beta_{l} - \hat{\beta_{l}} \Big) \epsilon_{j+|k|} \Bigg) \Bigg \|_{\mathbb{L}^{1}} \\
\leq \left \| \frac{1}{n} \sum_{j=1}^{n-|k|} \left( \sum_{l=1}^{p} x_{j,l} \left( \beta_{l} - \hat{\beta_{l}} \right) \sum_{l=1}^{p} x_{j+|k|,l} \left(  \beta_{l} - \hat{\beta}_{l} \right) \right) \right \|_{\mathbb{L}^{1}} + \left \| \frac{1}{n} \sum_{j=1}^{n-|k|} \left( \epsilon_{j} \sum_{l=1}^{p} x_{j+|k|,l} \left(  \beta_{l} - \hat{\beta}_{l} \right) \right) \right \|_{\mathbb{L}^{1}} \\
+ \left \| \frac{1}{n} \sum_{j=1}^{n-|k|} \left( \sum_{l=1}^{p} x_{j,l} \left( \beta_{l} - \hat{\beta_{l}} \right) \epsilon_{j+|k|} \right) \right \|_{\mathbb{L}^{1}},
\end{multline*}
then we swap the sums between them:
\begin{eqnarray*}
&\phantom{=}&
\left \| \frac{1}{n} \sum_{j=1}^{n-|k|} \left( \sum_{l=1}^{p} x_{j,l} \left( \beta_{l} - \hat{\beta_{l}} \right) \sum_{l=1}^{p} x_{j+|k|,l} \left(  \beta_{l} - \hat{\beta}_{l} \right) \right) \right \|_{\mathbb{L}^{1}} \\
&& +\: \left \| \frac{1}{n} \sum_{j=1}^{n-|k|} \left( \epsilon_{j} \sum_{l=1}^{p} x_{j+|k|,l} \left(  \beta_{l} - \hat{\beta}_{l} \right) \right) \right \|_{\mathbb{L}^{1}} + \left \| \frac{1}{n} \sum_{j=1}^{n-|k|} \left( \sum_{l=1}^{p} x_{j,l} \left( \beta_{l} - \hat{\beta_{l}} \right) \epsilon_{j+|k|} \right) \right \|_{\mathbb{L}^{1}} \\
&\leq & \left \| \frac{1}{n} \sum_{j=1}^{n-|k|} \left( \sum_{l=1}^{p} x_{j,l} \left( \beta_{l} - \hat{\beta_{l}} \right) \sum_{l=1}^{p} x_{j+|k|,l} \left(  \beta_{l} - \hat{\beta}_{l} \right) \right) \right \|_{\mathbb{L}^{1}} \\
&& +\: \left \| \frac{1}{n} \sum_{l=1}^{p} \sum_{j=1}^{n-|k|} \epsilon_{j} x_{j+|k|,l} \left(  \beta_{l} - \hat{\beta}_{l} \right) \right \|_{\mathbb{L}^{1}} + \left \| \frac{1}{n} \sum_{l=1}^{p} \sum_{j=1}^{n-|k|} \epsilon_{j+|k|} x_{j,l} \left( \beta_{l} - \hat{\beta_{l}} \right) \right \|_{\mathbb{L}^{1}},
\end{eqnarray*}
and using again the triangle inequality:
\begin{eqnarray*}
&\phantom{=}&
\left \| \frac{1}{n} \sum_{j=1}^{n-|k|} \left( \sum_{l=1}^{p} x_{j,l} \left( \beta_{l} - \hat{\beta_{l}} \right) \sum_{l=1}^{p} x_{j+|k|,l} \left(  \beta_{l} - \hat{\beta}_{l} \right) \right) \right \|_{\mathbb{L}^{1}} \\
&& +\: \left \| \frac{1}{n} \sum_{l=1}^{p} \sum_{j=1}^{n-|k|} \epsilon_{j} x_{j+|k|,l} \left(  \beta_{l} - \hat{\beta}_{l} \right) \right \|_{\mathbb{L}^{1}} + \left \| \frac{1}{n} \sum_{l=1}^{p} \sum_{j=1}^{n-|k|} \epsilon_{j+|k|} x_{j,l} \left( \beta_{l} - \hat{\beta_{l}} \right) \right \|_{\mathbb{L}^{1}} \\
&\leq & \left \| \frac{1}{n} \sum_{j=1}^{n-|k|} \left( \sum_{l=1}^{p} x_{j,l} \left( \beta_{l} - \hat{\beta_{l}} \right) \sum_{l'=1}^{p} x_{j+|k|,l'} \left(  \beta_{l'} - \hat{\beta}_{l'} \right) \right) \right \|_{\mathbb{L}^{1}} \\
&& +\: \frac{1}{n} \sum_{l=1}^{p} \left \| \sum_{j=1}^{n-|k|} \epsilon_{j} x_{j+|k|,l} \left(  \beta_{l} - \hat{\beta}_{l} \right) \right \|_{\mathbb{L}^{1}} + \frac{1}{n} \sum_{l=1}^{p} \left \| \sum_{j=1}^{n-|k|} \epsilon_{j+|k|} x_{j,l} \left( \beta_{l} - \hat{\beta_{l}} \right) \right \|_{\mathbb{L}^{1}}.
\end{eqnarray*}

For the first term of the right-hand side, we have:
\begin{multline*}
\left \| \frac{1}{n} \sum_{j=1}^{n-|k|} \left( \sum_{l=1}^{p} x_{j,l} \left( \beta_{l} - \hat{\beta_{l}} \right) \sum_{l'=1}^{p} x_{j+|k|,l'} \left(  \beta_{l'} - \hat{\beta}_{l'} \right) \right) \right \|_{\mathbb{L}^{1}} \\
= \left \| \frac{1}{n} \sum_{l=1}^{p} \sum_{l'=1}^{p} \sum_{j=1}^{n-|k|} x_{j,l} \left( \beta_{l} - \hat{\beta_{l}} \right) x_{j+|k|,l'} \left(  \beta_{l'} - \hat{\beta}_{l'} \right) \right \|_{\mathbb{L}^{1}},
\end{multline*}
then by triangle inequality:
\begin{multline*}
\left \| \frac{1}{n} \sum_{l=1}^{p} \sum_{l'=1}^{p} \sum_{j=1}^{n-|k|} x_{j,l} \left( \beta_{l} - \hat{\beta_{l}} \right) x_{j+|k|,l'} \left(  \beta_{l'} - \hat{\beta}_{l'} \right) \right \|_{\mathbb{L}^{1}} \\
\leq \frac{1}{n} \sum_{l=1}^{p} \sum_{l'=1}^{p} \left \| \sum_{j=1}^{n-|k|} x_{j,l} \left( \beta_{l} - \hat{\beta_{l}} \right) x_{j+|k|,l'} \left(  \beta_{l'} - \hat{\beta}_{l'} \right) \right \|_{\mathbb{L}^{1}}. 
\end{multline*}
Since $ab \leq \frac{1}{2} a^{2} + \frac{1}{2} b^{2}$, we get:
\begin{multline*}
\frac{1}{n} \sum_{l=1}^{p} \sum_{l'=1}^{p} \left \| \sum_{j=1}^{n-|k|} x_{j,l} \left( \beta_{l} - \hat{\beta_{l}} \right) x_{j+|k|,l'} \left(  \beta_{l'} - \hat{\beta}_{l'} \right) \right \|_{\mathbb{L}^{1}} \\
\leq \frac{1}{n} \sum_{l=1}^{p} \sum_{l'=1}^{p} \left \| \frac{1}{2} \sum_{j=1}^{n-|k|} \left( x_{j,l} \left( \beta_{l} - \hat{\beta_{l}} \right) \right)^{2} + \frac{1}{2} \sum_{j=1}^{n-|k|} \left( x_{j+|k|,l'} \left(  \beta_{l'} - \hat{\beta}_{l'} \right) \right)^{2} \right \|_{\mathbb{L}^{1}},
\end{multline*}
and by the triangle inequality:
\begin{multline*}
\frac{1}{n} \sum_{l=1}^{p} \sum_{l'=1}^{p} \left \| \frac{1}{2} \sum_{j=1}^{n-|k|} \left( x_{j,l} \left( \beta_{l} - \hat{\beta_{l}} \right) \right)^{2} + \frac{1}{2} \sum_{j=1}^{n-|k|} \left( x_{j+|k|,l'} \left(  \beta_{l'} - \hat{\beta}_{l'} \right) \right)^{2} \right \|_{\mathbb{L}^{1}} \\
\leq \frac{1}{2n} \sum_{l=1}^{p} \sum_{l'=1}^{p} \left \| \left( \beta_{l} - \hat{\beta_{l}} \right)^{2} \sum_{j=1}^{n-|k|} x_{j,l}^{2}  \right \|_{\mathbb{L}^{1}} + \frac{1}{2n} \sum_{l=1}^{p} \sum_{l'=1}^{p} \left \|  \left(  \beta_{l'} - \hat{\beta}_{l'} \right)^{2} \sum_{j=1}^{n-|k|} x_{j+|k|,l'}^{2}  \right \|_{\mathbb{L}^{1}}.
\end{multline*}

In conclusion, we have:
\begin{eqnarray*}
\left \| \hat{\gamma}_{k}^{\ast} - \hat{\gamma}_{k} \right \|_{\mathbb{L}^{1}}
&\leq & \frac{1}{2n} \sum_{l=1}^{p} \sum_{l'=1}^{p} \left \| \left( \beta_{l} - \hat{\beta_{l}} \right)^{2} \sum_{j=1}^{n-|k|} x_{j,l}^{2}  \right \|_{\mathbb{L}^{1}} + \frac{1}{2n} \sum_{l=1}^{p} \sum_{l'=1}^{p} \left \|  \left(  \beta_{l'} - \hat{\beta}_{l'} \right)^{2} \sum_{j=1}^{n-|k|} x_{j+|k|,l'}^{2}  \right \|_{\mathbb{L}^{1}} \\
&& +\: \frac{1}{n} \sum_{l=1}^{p} \left \| \sum_{j=1}^{n-|k|} \epsilon_{j} x_{j+|k|,l} \left(  \beta_{l} - \hat{\beta}_{l} \right) \right \|_{\mathbb{L}^{1}} + \frac{1}{n} \sum_{l=1}^{p} \left \| \sum_{j=1}^{n-|k|} \epsilon_{j+|k|} x_{j,l} \left( \beta_{l} - \hat{\beta_{l}} \right) \right \|_{\mathbb{L}^{1}}.
\end{eqnarray*}

\newpage

\bibliographystyle{abbrv}
\bibliography{articlebib}

\begin{thebibliography}{10}

\bibitem{anderson2011statistical}
T.~W. Anderson.
\newblock {\em The statistical analysis of time series}, volume~19.
\newblock John Wiley \& Sons, 2011.

\bibitem{bradley1985basic}
R.~C. Bradley.
\newblock Basic properties of strong mixing conditions.
\newblock In {\em Dependence in probability and statistics ({O}berwolfach,
  1985)}, volume~11 of {\em Progr. Probab. Statist.}, pages 165--192.
  Birkh\"auser Boston, Boston, MA, 1986.

\bibitem{brillinger2001time}
D.~R. Brillinger.
\newblock {\em Time series: data analysis and theory}.
\newblock SIAM, 2001.

\bibitem{brockwell2013time}
P.~J. Brockwell and R.~A. Davis.
\newblock {\em Time series: theory and methods}.
\newblock Springer Science \& Business Media, 2013.

\bibitem{dedecker1998central}
J.~Dedecker.
\newblock A central limit theorem for stationary random fields.
\newblock {\em Probability Theory and Related Fields}, 110(3):397--426, 1998.

\bibitem{dedecker2015optimality}
J.~Dedecker.
\newblock On the optimality of {M}cleish's conditions for the central limit
  theorem.
\newblock {\em Comptes Rendus Mathematique}, 353(6):557--561, 2015.

\bibitem{dedecker2015weak}
J.~Dedecker, H.~Dehling, and M.~S. Taqqu.
\newblock Weak convergence of the empirical process of intermittent maps in
  $\mathbb{L}_{2}$ under long-range dependence.
\newblock {\em Stochastics and Dynamics}, 15(02):1550008, 2015.

\bibitem{dedecker2010some}
J.~Dedecker, S.~Gou{\"e}zel, and F.~Merlevede.
\newblock Some almost sure results for unbounded functions of intermittent maps
  and their associated markov chains.
\newblock In {\em Annales de l'institut Henri Poincar{\'e} (B)}, volume~46,
  pages 796--821, 2010.

\bibitem{dmv2007weak}
J.~Dedecker, F.~Merlev{\`e}de, and D.~Voln{\`y}.
\newblock On the weak invariance principle for non-adapted sequences under
  projective criteria.
\newblock {\em Journal of Theoretical Probability}, 20(4):971--1004, 2007.

\bibitem{dedecker_prieur}
J.~Dedecker and C.~Prieur.
\newblock New dependence coefficients. examples and applications to statistics.
\newblock {\em Probability Theory and Related Fields}, 132(2):203--236, 2005.

\bibitem{gordin1969central}
M.~I. Gordin.
\newblock Central limit theorem for stationary processes.
\newblock {\em Doklady Akademii Nauk SSSR}, 188(4):739, 1969.

\bibitem{grenander2008statistical}
U.~Grenander and M.~Rosenblatt.
\newblock {\em Statistical analysis of stationary time series}, volume 320.
\newblock American Mathematical Soc., 2008.

\bibitem{hannan73clt}
E.~J. Hannan.
\newblock Central limit theorems for time series regression.
\newblock {\em Probability theory and related fields}, 26(2):157--170, 1973.

\bibitem{wuspectraldensity}
W.~Liu and W.~B. Wu.
\newblock Asymptotics of spectral density estimates.
\newblock {\em Econometric Theory}, pages 1218--1245, 2010.

\bibitem{liverani1999probabilistic}
C.~Liverani, B.~Saussol, and S.~Vaienti.
\newblock A probabilistic approach to intermittency.
\newblock {\em Ergodic theory and dynamical systems}, 19(03):671--685, 1999.

\bibitem{pipiras2017long}
V.~Pipiras and M.~Taqqu.
\newblock {\em Long-Range Dependence and Self-Similarity}.
\newblock Cambridge Series in Statistical and Probabilistic Mathematics.
  Cambridge University Press, 2017.

\bibitem{priestley1981spectral}
M.~B. Priestley.
\newblock Spectral analysis and time series.
\newblock 1981.

\bibitem{rio1999theorie}
E.~Rio.
\newblock {\em Th{\'e}orie asymptotique des processus al{\'e}atoires faiblement
  d{\'e}pendants}, volume~31.
\newblock Springer Science \& Business Media, 1999.

\bibitem{rosenblatt1956central}
M.~Rosenblatt.
\newblock A central limit theorem and a strong mixing condition.
\newblock {\em Proceedings of the National Academy of Sciences}, 42(1):43--47,
  1956.

\bibitem{rosenblatt2012stationary}
M.~Rosenblatt.
\newblock {\em Stationary sequences and random fields}.
\newblock Springer Science \& Business Media, 2012.

\bibitem{seneta2006regularly}
E.~Seneta.
\newblock {\em Regularly varying functions}, volume 508.
\newblock Springer, 2006.

\bibitem{wu2005nonlinear}
W.~B. Wu.
\newblock Nonlinear system theory: Another look at dependence.
\newblock {\em Proceedings of the National Academy of Sciences of the United
  States of America}, 102(40):14150--14154, 2005.

\bibitem{wudependence}
W.~B. Wu.
\newblock Asymptotic theory for stationary processes.
\newblock {\em Stat. Interface}, 4(2):207--226, 2011.

\end{thebibliography}
\addcontentsline{toc}{chapter}{Bibliographie}

\end{document}